\documentclass[onefignum,onetabnum]{siamart171218}

\usepackage{lipsum}
\usepackage{amsfonts}
\usepackage{graphicx}
\usepackage{epstopdf}
\usepackage{algorithmic}
\ifpdf
  \DeclareGraphicsExtensions{.eps,.pdf,.png,.jpg}
\else
  \DeclareGraphicsExtensions{.eps}
\fi

\newsiamremark{remark}{Remark}
\newsiamremark{hypothesis}{Hypothesis}
\crefname{hypothesis}{Hypothesis}{Hypotheses}
\newsiamthm{claim}{Claim}

\headers{Graphic Formulation of Non-isothermal CRN}{Z. Fang, A.J. van der Schaft, and C.H. Gao}

\author{Zhou Fang\thanks{Zhou Fang was with School of Mathematical Sciences, Zhejiang University, Hangzhou 310027, China, during this work.
  (\email{zhou\_fang@zju.edu.cn}) Currently, he is with the Department of Biosystems Science and Engineering, ETH Zurich.}
\and Arjan van der Schaft \thanks{Arjan van der Schaft is with Johann Bernoulli Institute, Faculty of Science and Engineering, University of Groningen, Netherlands. (\email{a.j.van.der.schaft@rug.nl}.)}
\and Chuanhou Gao \thanks{Chuanhou Gao is with School of Mathematical Sciences, Zhejiang University, Hangzhou 310027, China.
  (\email{gaochou@zju.edu.cn})}}

\usepackage{amsopn}

\makeatletter
\newcommand*{\addFileDependency}[1]{
  \typeout{(#1)}
  \@addtofilelist{#1}
  \IfFileExists{#1}{}{\typeout{No file #1.}}
}
\makeatother

\ifpdf
\hypersetup{
  pdftitle={Non-isothermal reaction networks},
  pdfauthor={Z. Fang, A.J. van der Schaft, and C.H. Gao}
}
\fi

\usepackage{amsmath}
\usepackage{mathrsfs,amssymb,amsmath,booktabs,array,xcolor,url,cite,color,soul,multirow,multicol}
\usepackage{mathbbold,bbm}
\usepackage{slashbox}
\usepackage{stfloats}
\usepackage{amsfonts}
\usepackage{extarrows}
\usepackage{textcomp}
\usepackage{cases}
\usepackage{bm}
\usepackage{arydshln}
\usepackage[all]{xy}
\usepackage{chemarrow}
\usepackage[version=4]{mhchem}

\newtheorem{condition}{Condition}
\newtheorem{example}{Example}

\def\dd{\text{d}}
\def\Exp{\textrm{Exp}}
\def\Ln{\textrm{Ln}}

\title{A graphic formulation of  non-isothermal chemical reaction systems and the analysis of detailed balanced networks
\thanks{Submitted to the editors DATE.
\funding{This work is in part supported by the National Natural Science Foundation of China under Grant No. 11671418 and 12071428, and in part by the Zhejiang Provincial Natural Science Foundation of China under Grant LZ20A010002.}}
}

\begin{document}

\maketitle

\begin{abstract}
	In this paper, we provide a graphic formulation of non-isothermal reaction systems and show that a non-isothermal detailed balanced network system converges (locally) asymptotically to the unique equilibrium within the invariant manifold determined by the initial condition.
	To model thermal effects, the proposed modeling approach extends the classical chemical reaction network by adding two parameters to each direct (reaction) edge, depicting, respectively, the instantaneous internal energy change after the firing of the reaction and the variation of the reaction rate with respect to the temperature.
	For systems possessing thermodynamic equilibria, our modeling approach provides a compact formulation of the dynamics where reaction topology and thermodynamic information are presented simultaneously.
	Finally, using this formulation and the Legendre transformation, we show that non-isothermal detailed balanced network systems admit some fundamental properties: dissipativeness, the detailed balancing of each equilibrium, the existence and uniqueness of the equilibrium, and the asymptotic stability of each equilibrium. 
	In general, the analysis and results of this work provide insights into the research of non-isothermal chemical reaction systems.
\end{abstract}



\section{Introduction}
Chemical reaction network (CRN) is a simple graphic structure that has been applied extensively to modeling and analyzing complex dynamic systems, ranging from biochemistry to epidemics. The corresponding theory, called CRN theory (CRNT), has been developed in over the last decades, since the seminal papers published in 1970's by Horn and Jackson \cite{horn1972general} and Feinberg \cite{feinberg1972complex}. One of the central issues in CRNT is to explore the connection between network topology and dynamical properties, including stability \cite{horn1972general,feinberg1972complex,rao2013graph,van2013mathematical,szederkenyi2011finding,ke2017realizations,al2016new,ke2019complex}, unique/multi-stationarity \cite{feinberg1995existence,feinberg1987chemical,feinberg1988chemical,craciun2006multiple,joshi2012simplifying}, robustness \cite{shinar2010structural,shinar2009sensitivity}, and persistence\footnote{persistence means there is no $\omega$-limit point on the boundary of the positive quadrant.} \cite{gopalkrishnan2014geometric,craciun2015toric,anderson2008global,anderson2011proof}.

Although CRNT has been very popular within biochemistry, it fails to handle chemical reactions occurring in the field of chemical engineering. The main reason is that CRNT assumes the temperature in the reaction system to be constant. However, many reactions absorb or release heat when they fire, leading to a change of temperature. A typical example is the ammonia synthesis reaction
$$\text{N}_2+3\text{H}_2\rightleftharpoons 2\text{NH}_3,$$
where $\text{N}_2$ is nitrogen, $\text{H}_2$ is hydrogen, and $\text{NH}_3$ is ammonia. The forward reaction will release the heat of $92.4$KJ/mol\footnote{KJ means kilojoule, energy unit; mol means mole, mass unit.} while the backward reaction will absorb the same amount of heat. For this class of reactions, called here non-isothermal CRNs (the classic CRNs are called isothermal CRNs for the sake of distinction), the dynamics exhibits more complex behaviors. On the one hand, the change of temperature should be captured, often represented by the change of internal energy (see \mbox{\cref{section notations}} for physical explanation of this terminology and others); on the other hand, the reaction rate coefficient is no longer a constant but a function of temperature following the Arrhenius law (or the transition state theory in a broader sense)\mbox{\cite{atkins2013elements,truhlar1983current,laidler1983development,eyring1935activated}}. Thus, new theory is needed to address the issues of modeling and analyzing dynamical behaviors of non-isothermal CRN systems.

A pioneering work in this regard is in Wang et al. \cite{wang2018port,wang2016irreversible}, investigating non-isothermal CRNs through the port-Hamiltonian modeling approach. This approach, rooted in the classical Hamiltonian equations, models the physical system by utilizing external ports and a Dirac structure. Mathematically, it amounts to define skew-symmetric tensor fields on the state space, an energy function called the Hamiltonian which is often given by the total stored energy of the system, and energy-dissipating relations. With this method, this work uses additional parameters to depict the Arrhenius law, write dynamics of isolated systems in the form of port-Hamiltonian systems, and express interconnection between target systems and environments via external ports. As a result it is shown that detailed balanced networks are dissipative with respect to availability functions of the entropy and convergent to a set of equilibria \cite[Theorem 5.1 and 5.2]{wang2018port}. However, apart from these results, many topics within the theory non-isothermal CRNs are not yet addressed. Up to now, there are still no systematic approaches to model generic non-isothermal CRN systems, especially when chemical reactions, mass fluxes, and heat exchanges are simultaneously present. 

For the above reasons, the current work intends to develop basic theory for non-isothermal CRNs, including the definition of basic concepts and the derivation of basic results that are parallel to those in isothermal CRNT. To this end, we provide a systematic method to model generic non-isothermal CRN system in order to derive results going beyond those in\mbox{\cite{wang2018port}}. Note that since non-isothermal CRN systems have to obey the first and second law of thermodynamics\mbox{\cite{ydstie2002passivity}}, thermodynamic knowledge will be key for their analysis. For the convenience of the readers, we present a review on thermodynamics in \mbox{\cref{section thermodynamics}} as preliminary knowledge. 

Actually, it is not a new idea to utilize thermodynamics to analyze CRN systems. Within the isothermal detailed balanced CRN framework, van der Schaft et al. \cite{van2013mathematical,van2013network} established a compact formulation which exhibits at the same time the graphic structure and the aforementioned thermodynamic knowledge, by which dynamical properties can be easily characterized. 
A non-isothermal analog to this compact formulation is presented in \cite{wang2018port} for isolated systems (i.e., systems with no mass flux or heat exchanges)  and greatly simplifies the stability analysis in that scenario.
Due to the introduction of pseudo-Helmholtz free energies, which served as Lyapunov functions for complex balanced networks \cite{horn1972general}, many studies focusing on the thermodynamic interpretations and properties of these functions have been published. 
Following \cite{van2013mathematical,van2013network}, Rao et al. \cite{rao2013graph} applied the pseudo-Helmholtz free energy to provide a similar compact mathematical formulation of complex-balanced networks, exhibiting at the same time the graphic structure and the thermodynamic relations. 
From the viewpoint of statistical mechanics, Anderson and his co-workers showed that the stationary distribution of a stochastic complex-balanced (isothermal) network has a product-form \cite{anderson2010product}, similar to Poisson distributions, and its non-equilibrium potential leads to a pseudo-Helmholtz free energy in the classical limit case \cite{anderson2015lyapunov}.
Inspired by these thermodynamic analyses, Fang and Gao \cite{fang2019lyapunov} introduced Lyapunov function PDEs based on the stationary distribution and showed their solution to be dissipative and be able to serve as Lyapunov functions for some special cases, including, but not limited to, complex-balanced networks.
The extension of classical thermodynamic concepts to non-equilibrium cases (especially to living systems) has been reviewed in \cite{beard2004thermodynamic}, providing the physico-chemical basis for analyzing large-scale metabolic networks in living systems (c.f. \cite{beard2005thermodynamic,beard2004thermodynamic}).
All these studies show good examples that thermodynamic knowledge can contribute to the analysis of CRN systems. In line with these studies, we here utilize isothermal CRNT and thermodynamics to construct the basic framework of non-isothermal CRNT, and further analyze the dynamical behaviors of non-isothermal detailed balanced CRN systems.
For the convenience of the reader, we also present a brief review on isothermal CRNT in \mbox{\cref{iso_CRNT}}.

The rest of this paper is organized as follows. In \cref{section modeling}, we establish a graphic formulation for the modeling of non-isothermal CRN systems. Then we provide a compact dynamic formulation of non-isothermal detailed balanced networks in \cref{section dissipativeness}, and further exhibit their dissipativeness. In \cref{section asymptotic stability}, we show the asymptotic stability of detailed balanced network systems by studying the Legendre transform of the availability function with the help of thermodynamics.
Finally, \cref{section discussion} concludes this paper. 
In the appendix, we list some mathematical notations and physical terminologies (see \cref{section notations}), review basic knowledge about thermodynamics and isothermal CRNT (see \cref{section appendix B}), and provide the detailed proofs of our main results (see \cref{section proof}). 



\section{Non-isothermal CRNs}{\label{section modeling}}
In this section, we modify the classical CRN structure so that it can cover non-isothermal reaction systems.
We refer readers to \cref{section appendix B} for a brief review of thermodynamics and the isothermal CRNT. 

\subsection{Dynamic model}
We consider the same network system as stated in \cref{iso_CRNT} but in the non-isothermal scenario. That is
\begin{equation*}
\alpha_{1,j} X_{1}+ 
\cdots + \alpha_{n,j} X_{n}  \ce{->[$k_j(T)$]}
\tilde\alpha_{1,j} X_{1}+ 
\cdots+ \tilde\alpha_{n,j} X_{n}, \quad j=1,...,r,
\end{equation*}
where $X_i$ represents the $i$-th species in the system, $\alpha_{i,j}$ and $\tilde\alpha_{i,j}$ are non-negative integers called stoichiometric coefficients, the integer vectors, $(\alpha_{1,j},\dots,\alpha_{n,j})^{\top}$ and $(\tilde \alpha_{1,j},\dots,\tilde \alpha_{n,j})^{\top}$, are termed as complexes, and $k_j(T)$ are reaction rate constants at the temperature $T$.
By the discussion in \mbox{\cref{section thermodynamics}}, the system can be fully described by internal energy and amounts of substances, i.e., $(U, N_1,\dots, N_n)$ or $(U,N)$, so in non-isothermal cases, we need to track the internal energy along with amounts of substances. For the $j$-th reaction, we denote the instantaneous energy change by $\Delta\mathcal{U}_j(U,N)$ (with the unit $J/\text{mol}$) and assume the reaction rate coefficient, $k_j(T)$, to follow Eyring's formula, i.e., \cite{eyring1935activated} 
\begin{equation}{\label{eq. erying's law}}
k_j(T)=\tilde{k}_{j}T \exp\left\{-\frac{\Delta G^{\ddagger}_j(T)}{RT}\right\}
= \tilde{k}_{j}T \exp\left\{-\frac{g^{\rm AS}_j(T)-y_{\sigma_{j}}^{\top}g(T)}{A_vRT}\right\}, 
\end{equation}
where $\tilde{k}_{j}$ is the reaction rate constant, $R$ is the Boltzmann constant, $\Delta G^{\ddagger}_j(T)$ (with the unit $J$) is the the Gibbs energy of activation, $A_v$ is the Avogadro constant, $g^{\rm AS}(T)$ is the free energy of the activated state, and $g(T)=\left(g_1(T),\dots,g_n(T)\right)^\top$ is the free energy vector for each individual molecules (see \cref{gGao}). 
Recall that the temperature $T$ is a function of $(U,N)$ (see \mbox{\cref{pro. fundamental equation})}, and, therefore, $k_j(T)$ is also determined by the system state $(U,N)$. In the sequel of this paper, we will drop the arguments of $T(U,N)$ to simply the notation; please keep in mind that $T$ is a function of $(U,N)$ rather than an independent variable or a fixed constant unless specified.

In line with the above facts, we can modify the above chemical reaction scheme by the following {\em pseudo-thermochemical equation}
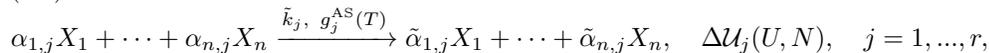
\begin{equation}{\label{eq. pseudo-thermochemical equation}}
\alpha_{1,j} X_{1}+ 
\cdots + \alpha_{n,j} X_{n} \ce{->[$\tilde k_j, ~ g^{\rm AS}_j(T)$]} 
\tilde\alpha_{1,j} X_{1}+ 
\cdots+ \tilde\alpha_{n,j} X_{n},   \quad  \Delta \mathcal{U}_j(U,N), \quad j=1,...,r,
\end{equation}
where mass variations are indicated by the stoichiometric coefficients, the instantaneous internal energy change is presented by $\Delta\mathcal{U}_j(U,N)$, and the transition state theory is depicted by constants $\tilde k_j$ and $g^{\rm AS}_j(T)$.
Note that \eqref{eq. pseudo-thermochemical equation} is different from the thermochemical equation, in which the variation of enthalpy under the standard condition\footnote{In chemistry, the standard condition means that the system is at one-atmosphere pressure, and the temperature is 298.15K (25$^{\circ}$C).} is presented instead of $\Delta\mathcal{U}_j(U,N)$.
Clearly, \eqref{eq. pseudo-thermochemical equation} and the thermochemical equation have different physical meanings, as a consequence of which we call \eqref{eq. pseudo-thermochemical equation} the pseudo-thermochemical equation.
In the sequel of this paper, we skip the arguments of the function $\Delta \mathcal{U}_j(U,N)$ to simplify notations; please keep in mind that $\Delta \mathcal{U}_j$ is a function rather than a constant term. Similar to the isothermal case, we can also view pseudo-thermochemical equations \eqref{eq. pseudo-thermochemical equation} as a graph, where complexes are vertexes, reaction arrows are edges, and $\Delta\mathcal{U}_j$, $\tilde k_j$, and $g^{\rm AS}_j(T)$ ($j=1,\dots,r$) are parameters associated with edges. Referring to \cref{def. isothermal crn}, we define non-isothermal CRN as follows. 

\begin{definition}[Non-isothermal CRN] {\label{def. non-isothermal crn}}
	A non-isothermal CRN is a quintuplet $(\mathcal{S},\mathcal{C}, \bar{\mathcal{R}},\bar{\mathcal K}, \mathcal G)$, where 
	\begin{itemize}
		\item  $\mathcal{S}\triangleq\{X_1,\dots, X_n\}$ is the species set representing the considered substances,
		\item $\mathcal{C}\triangleq\{y_1,\dots, y_m\}$ is the complex set consisting of all distinguished complexes where $m$ is the size of the set.
		\item $\bar {\mathcal{R}}\triangleq\left\{\left(y_{\sigma_{1}}\to y_{\pi_{1}}, \Delta\mathcal{U}_1 \right),\dots,\left(y_{\sigma_{r}}\to y_{\pi_{r}}, \Delta\mathcal{U}_r\right)\right\}$ is the reaction set with $\sigma_{j}$ and $\pi_{j}$ $(1\leq j\leq r)$ being indexes of the substrate complex and product complex of the $j$-th reaction. Particularly, for each reaction, $\sigma_{j}$ and $\pi_{j}$ should not be the same unless $y_{\sigma_{j}}$ and $y_{\pi_{j}}$ are both zero complexes.
		\item $\bar{\mathcal{K}}\triangleq(\tilde{k}_1,\dots,\tilde{k}_r)$. 
		\item $\mathcal{G}\triangleq (g^{\rm AS}_1(T),\dots,g^{\rm AS}_r(T))$.
	\end{itemize}
\end{definition}

Compared with the isothermal CRN $(\mathcal{S},\mathcal{C}, {\mathcal{R}},\mathcal K)$ (see \cref{def. isothermal crn}), the non-isothermal CRN attaches a new term, $\Delta\mathcal{U}_j$, to each reaction edge, while preserving other basic structures (the species set, the complex set, and connection relations).
Therefore, from the graph theory, the triplet $(\mathcal{S},\mathcal{C}, \bar {\mathcal{R}})$ that represents the skeleton of a non-isothermal CRN can be alternatively described by the complex matrix, $Y\triangleq(y_1,\dots,y_{m})\in\mathbb{R}^{n\times m}$, the incidence matrix $D$ of the form
\begin{equation*}
D_{i,j}\triangleq\left\{
\begin{array}{cc}
1,   & i=\pi_{j},  \\
-1,  & i=\sigma_{j}, \\
0,   & \text{otherwise},
\end{array}
\right.
\quad i\in\{1,\dots,m\} \text{~and~} j\in\{1,\dots,r\},
\end{equation*}
and the energy variation matrix $\Delta \mathcal U\triangleq\left(\Delta \mathcal U_1,\dots,\Delta \mathcal U_r\right)$.
Also, note that though ${\mathcal{K}}$ and $\bar{\mathcal{K}}$ both consist of constants, their entries have different physical meanings (see \eqref{eq. erying's law}); specifically, $\tilde k_j$ is a ``part" of $k_j$. 

Following mass-action kinetics, the reaction rate can be expressed as
\begin{equation}{\label{eq. transition state theory}}
v_{j}(U,N) = \tilde{k}_{j}T \exp\left\{-\frac{g^{\rm AS}_j(T)-y_{\sigma_{j}}^{\top}g(T)}{A_vRT}\right\} \Exp{\left(y_{\sigma_{j}}^{\top} \Ln N\right)},
\quad j=1,\dots,r.
\end{equation}
Imitating the dynamic model in isothermal CRN, i.e., \cref{eq. dynamics of isothermal CRN}, we can write the current one to be
\begin{equation}{\label{eq. dynamics of non-isothermal CRNs}}
\left(
\begin{array}{c}
\dot U \\  \dot N 
\end{array}
\right)
= 
\left(
\begin{array}{c}
\Delta \mathcal U \\ 	\Gamma 
\end{array}
\right) v(U,N),
\end{equation}
where $v(U,N)=\left(v_{1}(U,N),\dots,v_{r}(U,N)\right)^{\top}$. Note that we write the reaction rate as a function of $(U,N)$ in the above two equations, because $T$ is a function of $(U,N)$.

In this paper, we concern ourselves with analyzing the dynamic behaviors of the non-isothermal reaction system \eqref{eq. dynamics of non-isothermal CRNs}.
Particularly, we will show a non-isothermal analogy to \cref{thm. isothermal detailed balanced}. To this end, some assumptions related to the energy variation matrix $\Delta \mathcal U $ and the function set $\mathcal G$ should be made from the viewpoint of thermodynamics, which enables the network to describe open systems and paves the way for further analyses.

\subsection{Thermodynamic constraints on $\Delta \mathcal U $ and $\mathcal G$}\label{subsection thermodynamic constrains on U and G}
From the viewpoint of thermodynamics, for different network systems, such as isolated systems, isothermal systems, and those in-between, the matrix $\Delta \mathcal{U} $ and set $\mathcal{G}$ are different. In this subsection, we utilize thermodynamic knowledge (see \mbox{\cref{section thermodynamics}}) to put some constraints on them according to different systems.

For isolated systems where no energy or mass flux takes place at the boundary, the internal energy is conserved based on the first law of thermodynamics. Therefore, each reaction in this case has $\Delta\mathcal{U}_j=0$. 
Yet for isothermal systems, to keep the temperature fixed, each reaction in the form of \cref{eq. pseudo-thermochemical equation} admits $$\Delta\mathcal{U}_j= \left(y_{\pi_{j}}-y_{\sigma_{j}}\right)^{\top} u(T_e),$$ where $T_e$ is the fixed temperature (usually the environment temperature), and $u(\cdot)$ is the vector of energies of individual molecules (see \eqref{eq. energies of individual molecules}).

Isolated systems and isothermal systems are two kinds of extreme systems, where heat exchanges are, respectively, zero and infinitely fast. 
For a system in-between, we can model the heat exchange process by a pseudo-thermochemical equation
\begin{equation*}{\label{eq. heat exchange by pseudo-thermochemical equations}}
\emptyset \ce{->[$ h, ~ A_vRT\ln T$]}  \emptyset, \quad  T_e-T,
\end{equation*} 
where $\emptyset$ is the linear combination of substances with respect to a zero vector, $h$ is a constant depending on heat fluxes, and $T_e$ is the environmental temperature.
The validity of this modeling approach is shown as follows.
Such an artificial reaction takes place at the rate of $h$ and contributes to a change of the internal energy by $T_e-T $ after each firing (see \eqref{eq. dynamics of non-isothermal CRNs} and \eqref{eq. erying's law}).
Consequently, its effect is equivalent to the Fourier's law that describes the process of heat exchanges between two surfaces in contact.
Notably, the above pseudo-thermochemical equation requires a non-isothermal CRN to involve reactions having at the same time a zero substrate complex and a zero product complex (c.f. \cref{def. non-isothermal crn}), which is not the case in isothermal CRN. Also, for this system the mass fluxes can be modeled by the following pseudo-thermochemical equations
\begin{equation*}
\emptyset \ce{->[$\tilde q^{\text i{in}}_i, ~ A_vRT\ln T $]}  X_i,  \quad u_i(T_e) \quad \text{and} \quad 
X_i \ce{->[$\tilde q^{\text {out}}_i, ~ A_vRT\ln T +g_i(T) $]} \emptyset,  \quad -u_i(T),
\end{equation*} 
where the first ``reaction" brings one molecule of the $i$-th substance and $u_i(T_e)$ energy to the system at the rate of $\tilde q^{\text {in}}_i$, and the second reaction remove one molecule of the $i$-th substance and $u_i(T)$ energy at the rate of $\tilde q^{\text {out}}_i N_i$.

As far as non-isothermal CRN systems are concerned, the instantaneous internal energy change $\Delta\mathcal{U}_j$ of the $j$th ($j=1,..,r$) real chemical reactions should be zero due to the first law of thermodynamics, which suggests the internal energy to change only in cases when mass fluxes or heat exchanges take place.

To emphasize different natures of the above-mentioned processes, we denote three reaction subsets, $\bar {\mathcal{R}}_{\rm CR}$, $\bar {\mathcal{R}}_{\rm IO}$, and $\bar {\mathcal{R}}_{\rm HE}$, as follows, depicting, respectively, real chemical reactions, mass fluxes, and heat exchanges. 
\begin{align}
\bar {\mathcal{R}}_{\rm CR} \triangleq 
\big\{ &  \left(y_{\sigma_{j}}\to y_{\pi_{j}}, \Delta\mathcal{U}_j \right)\in \bar{\mathcal{R}} ~\big|~ y_{\sigma_{j}}\neq \mathbbold{0}_{n},~y_{\pi_{j}}\neq \mathbbold{0}_{n}, \text{~and~ $\Delta\mathcal{U}_j$ is constant} \big\} \notag\\
\bar {\mathcal{R}}_{\rm IO} \triangleq  \notag
\big\{ & 	\left(y_{\sigma_{j}}\to y_{\pi_{j}}, \Delta\mathcal{U}_j \right)\in \bar{\mathcal{R}} ~\big|   \\
& \notag \text{Each element is either}
\left(\mathbbold{0}_n \to \delta_{i}, u(T_e)\right) \text{~or~} \left(\delta_{i} \to \mathbbold{0}_n, -u(T)\right),~i=1,\dots,n.\big\} \\
\bar {\mathcal{R}}_{\rm HE} \triangleq  \notag
\big\{ &
\left(y_{\sigma_{j}}\to y_{\pi_{j}}, \Delta\mathcal{U}_j \right)\in \bar{\mathcal{R}} ~\big|~ 
y_{\sigma_{j}} = y_{\pi_{j}}= \mathbbold{0}_{n}, \text{ and } \Delta\mathcal{U}_j=T_e-T
\big\} 
\end{align}
Denote the sizes of the above reaction subsets by $r_{\rm CR}$, $r_{\rm IO}$ and $r_{\rm HE}$, respectively.
Therefore, the above-discussed properties of the energy variation matrix $\Delta \mathcal U$ and the function set $\mathcal G$ can be concluded by the following assumptions.
\begin{condition}{\label{condition. 2}}
	$\bar {\mathcal{R}}=\bar {\mathcal{R}}_{\rm CR}\bigcup \bar {\mathcal{R}}_{\rm IO}\bigcup\bar {\mathcal{R}}_{\rm HE}$.
	For simplicity, we assume 
	\begin{equation*}
	\left(y_{\sigma_{j}}\to y_{\pi_{j}}, \Delta\mathcal{U}_j \right) \in \left \{
	\begin{array}{ll}
	\bar {\mathcal{R}}_{\rm CR}, & 0< j \leq r_{\rm CR}, \\
	\bar {\mathcal{R}}_{\rm IO}, &  r_{\rm CR}< j \leq r_{\rm CR} + r_{\rm IO}, \\
	\bar {\mathcal{R}}_{\rm IO}, & r_{\rm CR} + r_{\rm IO}< j \leq r. 
	\end{array}
	\right.
	\end{equation*}
\end{condition}

\begin{condition}\label{condition. 3}
	One of the following statements about $\Delta\mathcal{U}_j $ is true.
	\begin{itemize}
		\item For each reaction in $\bar {\mathcal{R}}_{\rm CR}$, we have $\Delta\mathcal{U}_j(U,N) =0$.
		\item For each reaction in $\bar {\mathcal{R}}_{\rm CR}$, we have $\Delta\mathcal{U}_j(U,N)= \left(y_{\pi_{j}}-y_{\sigma_{j}}\right)^{\top} u(T_e)$ and $r_{\rm IO}=r_{\rm HE}=0$.
	\end{itemize}
\end{condition}

\begin{condition}{\label{condition. 4}}
	$\Delta G_j^{\ddagger}=RT\ln T$ for each reaction in $\bar{\mathcal{R}}_{IO}$ and $\bar{\mathcal{R}}_{HE}$.
	In other words, if $y_{\sigma_{j}}=\mathbbold{0}_n$, then $g^{\rm AS}_{j}(T)=A_vRT\ln T$;
	if $y_{\sigma_{j}}=\delta_{i}$ and $y_{\pi_{j}}=\mathbbold{0}_n$, then 
	$g^{\rm AS}_{j}(T)=A_vRT\ln T+g_{i}(T)$.
\end{condition}

\subsection{Invariant manifolds}

Provided with \cref{condition. 2}, we can rewrite the incidence matrix as $$D=\left(D_{\rm CR},D_{\rm IO},D_{\rm HE}\right)$$
and the energy variation matrix as
$$\Delta \mathcal U(U,N)=\left(\Delta \mathcal U_{\rm CR}(U,N), \Delta \mathcal U_{\rm IO}(U,N),\Delta \mathcal U_{\rm HE}(U,N)\right),$$
where $D_{\ell}$ and $\Delta \mathcal U_{\ell}$ are, respectively, the incidence matrix and the energy variation matrix with respect to the subnetwork $\bar {\mathcal{R}}_{\ell}$ ($\ell\in\{\rm CR, IO, HE \}$).
Therefore, we can divide the vector field of \eqref{eq. dynamics of non-isothermal CRNs} into three parts as follows, which represent the effects of real chemical reactions, mass fluxes, and heat exchanges, respectively,
\begin{align}{\label{eq. dynamics of non-isothermal crn splited}}
\left(
\begin{array}{c}
\dot U \\ 	\dot N 
\end{array}
\right)
=& 
\left(
\begin{array}{c}
\Delta \mathcal U_{\rm CR} \\ YD_{\rm CR}
\end{array}
\right) v_{CR}(U,N)
+
\left(
\begin{array}{c}
\Delta \mathcal U_{\rm IO} \\ YD_{\rm IO} 
\end{array}
\right) v_{IO}(U,N) 
+
\left(
\begin{array}{c}
\Delta \mathcal U_{\rm HE} \\YD_{\rm HE}  
\end{array}
\right) v_{HE}(U,N). 
\end{align}
Here, $v_{CR}(U,N)=\left(v_{1}(U,N),\dots,v_{r_{CR}}(U,N)\right)^{\top}$ depicts reaction rates of $\bar{\mathcal{R}}_{\rm CR}$,
$v_{IO}(U,N)=\left(v_{r_{CR}+1}(U,N),\dots,v_{r_{CR}+r_{IO}}(U,N)\right)^{\top}$ depicts reaction rates of $\bar{\mathcal{R}}_{\rm IO}$,
and $v_{HE}(U,N)=\left(v_{r_{CR}+r_{IO}+1}(U,N),\dots,v_{r}(U,N)\right)^{\top}$ depicts reaction rates of $\bar{\mathcal{R}}_{\rm HE}$.

Similar to the stoichiometric matrix $\Gamma=YD$ for the isothermal case, we define the stoichiometric-like matrix as follows to indicate the potential change of the state,
\begin{equation}{\label{eq. invariant manifold}}
\tilde{ \Gamma} \triangleq 
\left\{
\begin{array}{ll}
\left(
\begin{array}{c}
\Delta \mathcal U \\ \Gamma 
\end{array}
\right),  
& \text{if $\Delta \mathcal U$ is a constant matrix,} \\
\left(
\begin{array}{cc}
\mathbbold{0}^{\top}_n & 1\\  \Gamma& \mathbbold{0}_n  
\end{array}
\right),
& \text{otherwise.}
\end{array}
\right.	
\end{equation}
From \eqref{eq. dynamics of non-isothermal CRNs} (also \eqref{eq. dynamics of non-isothermal crn splited}), we can observe that the increment of the state always belongs to the linear space $\text{Im} \tilde{ \Gamma}$, and, therefore, can only evolve in an invariant set $$\left((U(0),N(0))^{\top}+\text{Im} \tilde{ \Gamma}\right)\bigcap\mathbb{R}^{n+1}_{\geq 0} \bigcap \{(U,N)~|~U\geq N^{\top}u(0)\},$$
where the inequality $U\geq N^{\top}u(0)$ suggests the temperature to be non-negative. 
Parallel to the isothermal case, we term the linear space $\text{Im} \tilde{ \Gamma}$ as the \textit{stoichiometric-like subspace}, the above invariant sets as the \textit{stoichiometric-like compatibility class}, and the interior of the invariant set,
\begin{equation}
{\mathcal{PS}}(U(0),N(0))\triangleq\left((U(0),N(0))^{\top}+\text{Im} \tilde{ \Gamma}\right)\bigcap\mathbb{R}^{n+1}_{> 0} \bigcap \left\{(U,N)^{\top}~|~U> N^{\top}u(0)\right\},
\end{equation}
as the \textit{positive stoichiometric-like compatibility class}.

\begin{remark}\label{rem_3.2}
	Compared with the work in \cite{wang2016irreversible,wang2018port}, our model $(\mathcal{S},\mathcal{C}, \bar{\mathcal{R}}, \bar{\mathcal{K}}, \mathcal{G})$ utilizes the transition state theory to provide a much clearer interpretation to reaction kinetics and a close connection between the kinetics and thermodynamic properties (c.f. \eqref{eq. transition state theory} and \cref{condition. 4}), which enables the modeling approach to cover a broader class of reaction systems (see \cref{subsection thermodynamic constrains on U and G}) and also greatly benefit the analysis of dynamical behaviors in the sequel of this paper.
	In the meanwhile, our approach can also provide a clear statement of the invariant manifold of the considered dynamic system, which is also essential in the analysis of asymptotic stability.
\end{remark}

\begin{example}{\label{exmp. 2}}
	We consider an isolated chemical reaction system having the following pseudo-thermochemical equations 
	\begin{equation}{\label{eq. example pseudo-thermochemical equations}}
	X_1+X_2  \ce{->[$\tilde k_1, ~ g^{\rm AS}_1(T)$]}  X_3,   \quad 0; \qquad X_3 \ce{->[$\tilde k_2, ~ g^{\rm AS}_2(T)$]}  X_1+X_2,  \quad 0.
	\end{equation}
	Alternatively, the pseudo-thermochemical equation can also be presented by a non-isothermal CRN $(\mathcal{S},\mathcal{C}, \bar{\mathcal{R}},\bar{\mathcal K}, \mathcal G)$ where $\mathcal{S}=\{X_1,X_2,X_3\}$, $\mathcal{C}=\{(1,1,0)^{\top},(0,0,1)^{\top}\}$, $$\bar{\mathcal{R}}=\left\{
	\bigg((1,1,0)^{\top} \to (0,0,1)^{\top},0 \bigg),
	\bigg((0,0,1)^{\top} \to (1,1,0)^{\top},0 \bigg)
	\right\},$$
	$\bar{\mathcal K}=\left(\tilde k_{1}, \tilde k_{2}\right)$ and $\mathcal{G}=\left(g^{\rm AS}_{1}(T),g^{\rm AS}_{2}(T)\right)$.
	Note that, in the above network, both reactions in $\bar{\mathcal{R}}$ belong to $\bar{\mathcal{R}}_{\rm CR}$, i.e., both reactions are ``real" chemical reactions; also \cref{condition. 3} and \cref{condition. 4} are clearly satisfied.
	In this network, the incidence matrix $D$ and the complex matrix $Y$ are given by
	\begin{equation*}
	D=
	\left(
	\begin{array}{cc}
	-1 & 1\\
	1 & -1
	\end{array}
	\right),
	\qquad\qquad
	Y=
	\left(
	\begin{array}{cc}
	1 & 0\\
	1 & 0 \\
	0 & 1
	\end{array}
	\right).
	\end{equation*}
	Finally by \eqref{eq. dynamics of non-isothermal CRNs} or \eqref{eq. dynamics of non-isothermal crn splited}, we can express the dynamics as
	\begin{align}
	\notag 
	\left(
	\begin{array}{c}
	\dot U \\ \dot N 
	\end{array}
	\right)
	&=
	\left(
	\begin{array}{c}
	\Delta \mathcal U \\	\Gamma 
	\end{array}
	\right) v(U,N) \\
	&= 
	\underbrace{\left(
		\begin{array}{cc}
		0 & 0 \\
		-1 & 1  \\
		-1 & 1\\
		1 & -1  \\
		\end{array}
		\right)}_{ \tilde \Gamma}
	\left(
	\begin{array}{cc}
	\tilde k_{1} T \exp\left\{-\frac{g^{\rm AS}_j(T)-g_1(T)-g_2(T)}{A_vRT}\right\} N_1 N_2   \\
	\tilde k_{2} T \exp\left\{-\frac{g^{\rm AS}_j(T)-g_3(T)}{A_vRT}\right\}N_3
	\end{array}
	\right), \notag
	\end{align}
	where the state of the system can only evolve in the stoichiometric-like compatibility class $\left((U(0),N(0))^{\top}+\text{Im} \tilde{ \Gamma}\right)\bigcap\mathbb{R}^{4}_{\geq 0} \bigcap \{(U,N)~|~U\geq N^{\top}u(0)\}$.
\end{example}

\section{Non-isothermal detailed balanced CRNs and their dissipativeness}{\label{section dissipativeness}}
Recall that our goals is to extend basic results of detailed balanced system (see \cref{thm. isothermal detailed balanced}) to the non-isothermal case.
In this section, we define detailed balanced networks for non-isothermal CRNs and establish a compact dynamic formula for them by which strictly dissipativeness can be shown. 
The stability result and detailed balancing of each equilibrium are respectively shown in \cref{thm. stability} and \cref{cor. detailed balancing of each equilibrium}, which extend the first two results of \cref{thm. isothermal detailed balanced}.
The third extension regarding the asymptotic convergence to a unique positive equilibrium will be discussed in the next section.

\subsection{Non-isothermal detailed balanced networks}
The concept of detailed balance consists of two key components, the reversibility and the balance of reaction rates of each reversible pair at an equilibrium.
We first define the reversibility of non-isothermal CRNs. 
\begin{definition}[Reversibility]
	We call a non-isothermal CRN $(\mathcal{S},\mathcal{C}, \bar{\mathcal{R}}, \bar{\mathcal{K}}, \mathcal{G})$ reversible, if for every reaction $(y_{\sigma_{j}}\to y_{\pi_{j}}, \Delta\mathcal{U}_{j})\in \bar{\mathcal{R}}$ there exists a unique $(y_{\sigma_{\tilde j}}\to y_{\pi_{\tilde j}}, \Delta\mathcal{U}_{\tilde j})\in \bar{\mathcal{R}}$ such that $y_{\sigma_{j}}=y_{\pi_{\tilde j}}$ and $y_{\pi_{j}}=y_{\sigma_{\tilde j}}$.
	Moreover, we name $(y_{\sigma_{\tilde j}}\to y_{\pi_{\tilde j}}, \Delta\mathcal{U}_{\tilde j})$ the backward/reverse reaction of  $(y_{\sigma_{j}}\to y_{\pi_{j}}, \Delta\mathcal{U}_{j})$.
\end{definition}

In both isothermal case and non-isothermal case, reversibility means that each reaction has a backward counterpart that switches the substrate complex and the production complex of the former reaction. (Please refer to \cref{iso_CRNT} for the reversibility of a isothermal CRN.)
However, different from the case in isothermal CRNs where reversible networks must have an even number of reactions, reversible non-isothermal CRNs can have an odd number of reactions, as it can have a reaction $(\mathbbold{0}_n\to\mathbbold{0}_n, \Delta\mathcal{U}_j)$ (c.f. \cref{def. non-isothermal crn}) whose reverse reaction is itself. Provided with \cref{condition. 2}, we can also observe that each pair of forward and backward reactions belong to the same reaction subset, that is if $(y_{\sigma_{j}}\to y_{\pi_{j}}, \Delta\mathcal{U}_{j})\in \bar{\mathcal{R}}_{\ell}$ then its backward reaction also belongs to $ \bar{\mathcal{R}}_{\ell}$ ($\ell\in\{\rm CR, IO, HE \}$). For reactions in $\bar{\mathcal{R}}_{\rm CR}$, we further assume the activated states of forward and backward reactions in each pair to be identical.

\begin{condition}{\label{condition. 5}}
	For each pair of forward and backward chemical reactions, $(y_{\sigma_{j}}\to y_{\pi_{j}}, \Delta\mathcal{U}_{j}(U,N))$ and $(y_{\sigma_{\tilde j}}\to y_{\pi_{\tilde j}}, \Delta\mathcal{U}_{\tilde j}(U,N))$, in $\bar{\mathcal{R}}_{\rm CR}$, there is the relation $g^{\rm AS}_{j}(T)=g^{\rm AS}_{\tilde j}(T)$ for all $T>0$.
\end{condition}

From the viewpoint of thermodynamics, \cref{condition. 5} means that the activated states of forward and backward reactions need to cross the same energy barrier to fire reactions (c.f. \cite{atkins2013elements,kondepudi2014modern}). This condition will play an important role in establishing a compact dynamic equation that exhibits at the same time the graphic topology and thermodynamic information, which will be discussed in the next subsection.

We then define detailed balanced networks for non-isothermal CRNs by requiring reaction rates and instantaneous energy changes of reversible reactions to be balanced at some equilibrium.

\begin{definition}[Non-isothermal detailed balanced networks]{\label{def. non-isothermal detailed balanced network}}
	A reversible non-isothermal CRN $(\mathcal{S},\mathcal{C}, \bar{\mathcal{R}}, \bar{\mathcal{K}}, \mathcal{G})$ is detailed balanced, if at some state $(U^*,N^*)$ each pair of forward and backward reactions, $(y_{\sigma_{j}}\to y_{\pi_{j}}, \Delta\mathcal{U}_{j})$ and $(y_{\sigma_{\tilde j}}\to y_{\pi_{\tilde j}}, \Delta\mathcal{U}_{\tilde j})$, satisfy
	\begin{itemize}
		\item $v_{j}(U^*,N^*)=v_{\tilde j}(U^*,N^*)$, i.e.,  reaction rates are balanced,
		\item  $\Delta \mathcal{U}_{j}(U^*,N^*)+\Delta \mathcal{U}_{\tilde j} (U^*,N^*) =0$, i.e., energy changes are balanced. 
	\end{itemize}
	Moreover, we name such state $(U^*,N^*)$ a detailed balanced equilibrium.
\end{definition}

Compared with the corresponding definition for isothermal case (see \cref{iso_CRNT}), non-isothermal detailed balanced networks require one more condition that the energy changes of each pair of reversible reactions are balanced.
This requirement follows from the philosophy that the system is described by the internal energy and mass, and, therefore,
a detailed balanced equilibrium should balance both the mass changes (by the restriction on reaction rates) and internal energy changes (by the restriction on instantaneous energy changes).
The second condition of \cref{def. non-isothermal detailed balanced network} is non-trivial for open systems where $\Delta \mathcal {U}_j$ can vary with respect to the temperature (see the definition of $\bar{\mathcal{R}}_{\rm IO}$) but trivial for isolated or isothermal systems in which $\Delta \mathcal {U}_j$ are always constants. 
Moreover, to detailedly balance the pair $\left(\mathbbold{0}_n \to \delta_{i}, u(T_e)\right)$ and 
$\left(\delta_{i} \to \mathbbold{0}_n, -u(T)\right)$ (or, $(\mathbbold{0}_{n},\mathbbold{0}_{n}, T_e-T)$ and itself), a detailed balanced equilibrium must satisfy $T_e=T^*$, where $T^*$ is the system temperature at the sate $(U^*, N^*)$. 
In other words, a detailed balanced equilibrium satisfies
\begin{equation}{\label{eq. temperature at detailed balanced equilirbium}}
T_e=T^*, \quad \text{if } r_{\rm IO}+r_{\rm HE}\neq 0.
\end{equation}

\begin{remark}
	For isolated isothermal reversible CNRs, Wegscheider's identity\mbox{\cite{wegscheider1902simultane}} reveals the connection between the existence of a detailed balanced equilibrium and the kinetic parameters (only including $\mathcal{K}$) of chemical reactions. Here, based on \mbox{\cref{condition. 5}}, we also recur this identity for isolated non-isothermal reversible CRNs, i.e., $\bar{\mathcal R}=\bar{\mathcal{R}}_{\rm CR}$. Assume that the considered reversible network follows mass-action kinetics and Eyring's law \mbox{\eqref{eq. erying's law}}, then the existence of a detailed balanced equilibrium at a given temperature $T^*$ is equivalent to the validity of Wegscheider's identity\mbox{\cite{gorban2011extended}}
	\begin{equation*}
	\prod_{j=1}^{r} \left( k_j (T^*)\right)^{\lambda_j} 
	=
	\prod_{j=1}^{r} \left( k_{\tilde j} (T^*)\right)^{\lambda_{\tilde j}}, 
	\qquad \forall \lambda\triangleq \left(\lambda_1,\dots, \lambda_r\right)\in \mathbb R^{r}, \text{ s.t. } \Gamma \lambda = \mathbbold{0}_{n},
	\end{equation*}
	where $\tilde j$ is the index of the reverse reaction of the $j$-th reaction. Further utilizing \mbox{\eqref{eq. erying's law} and \mbox{\cref{condition. 5}}}, the above equality can be rewritten as 
	\begin{equation}{\label{eq. pseudo-Wegscheider's identy}}
	\prod_{j=1}^{r} \left( \tilde k_j\right)^{\lambda_j} 
	=
	\prod_{j=1}^{r} \left( \tilde k_{\tilde j} \right)^{\lambda_{\tilde j}}, 
	\qquad \forall \lambda\triangleq \left(\lambda_1,\dots, \lambda_r\right)\in \mathbb R^{r}, \text{ s.t. } \Gamma \lambda = \mathbbold{0}_{n},
	\end{equation}
	which suggests that detailed balanced equilibria at different temperature levels are present or absent simultaneously. 
	Particularly, if $\tilde k_j=\tilde k_{\tilde j}$ holds for any reversible reaction pair, the considered system possesses infinite many detailed balanced equilibria. 
\end{remark}

We use the following example to exhibit the detailed balanced equilibrium in non-isothermal CRNs.

\begin{example}\label{exmp. 3}
	We still consider the non-isothermal system introduced in \cref{exmp. 2} whose pseudo-thermochemical equations follow \eqref{eq. example pseudo-thermochemical equations}.
	Obviously, the non-isothermal CRN is reversible as the two chemical reactions switch each other's substrate complex and product complex. 
	Provided with \cref{condition. 5}, i.e., $g^{AS}_1(T)=g^{AS}_2(T)$, we can observe that \eqref{eq. pseudo-Wegscheider's identy} is satisfied and find a detailed balanced equilibrium $(U^*,N^*)$ with
	\begin{equation*}
	N^*=\left(\exp\left\{-\frac{g_1(1)}{A_v R}\right\}, \exp\left\{-\frac{g_2(1)}{A_v R}\right\}, \frac{\tilde{k}_1}{\tilde{k}_2}\exp\left\{-\frac{g_3(1)}{A_v R}\right\}\right)^{\top}
	~\text{and} ~~
	U^*={N^*}^{\top}u(1),
	\end{equation*}
	which balances both the reaction rates and instantaneous energy changes (c.f. \eqref{eq. transition state theory} and \eqref{eq. pseudo-thermochemical equation}).
	As a result, the network is detailed balanced.
\end{example}

\subsection{A compact formula of detailed balanced networks}
Borrowing the idea and techniques in \cite{van2013mathematical,wang2018port}, we provide, in this subsection, non-isothermal detailed balanced networks with compact dynamic equations that exhibit at the same time network structures and thermodynamic information.
The key to providing such a formula is to rewrite reaction rates using thermodynamic quantities.

As discussed in the previous subsection, a pair of forward and backward reactions must belong to the same reaction subset provided with \cref{condition. 2}.
Therefore, without loss of generality, we assume the $(2j-1)$-th reaction and $2j$-th reaction ($j=1,\dots,\frac{r_{\rm CR}+r_{\rm IO}}{2}$) in a reversible network $(\mathcal{S},\mathcal{C}, \bar{\mathcal{R}}, \bar{\mathcal{K}}, \mathcal{G})$ are a pair of forward and backward reactions.
Then, we define a matrix $B \in \mathbb{R}^{m\times (r_{\rm CR}+r_{\rm IO})/2}$ by
\begin{equation}\label{eq. definition of matrix B}
B_{i,j}\triangleq\left\{
\begin{array}{cc}
1,   & i=\pi_{2j-1},  \\
-1,  & i=\sigma_{2j-1}, \\
0,   & \text{otherwise},
\end{array}
\right.
\quad i\in\{1,\dots,m\} \text{~and~} j\in\{1,\dots,(r_{\rm CR}+r_{\rm IO})/2\}.
\end{equation}
depicting the connection relation in a reversible network.
Moreover, we term $B_{\rm CR}\in \mathbb{R}^{m\times r_{\rm CR}/2}$ as the subblock of the matrix $B$ consisting of the first $(r_{\rm CR}/2)$ columns, and $B_{\rm IO}\in \mathbb{R}^{m\times r_{\rm IO}/2}$ as the subblock of the matrix $B$ consisting of the last $(r_{\rm IO}/2)$ columns. $B_{\rm CR}$ and $B_{\rm IO}$ depict respectively the connection relation in $\bar{\mathcal{R}}_{\rm CR}$ and $\bar{\mathcal{R}}_{\rm IO}$.
In a reversible network, though $B_{\ell}$ is only half of the scale of the incidence matrix $D_{\ell}$ ($\ell\in\{\rm CR, IO\}$), both matrices can represent the topology of $\bar{\mathcal{R}}_{\ell}$ and contain the same information.

\begin{proposition}{\label{prop. stoichiometric compatibility class}}
	In a reversible non-isothermal CRN $(\mathcal{S},\mathcal{C}, \bar{\mathcal{R}}, \bar{\mathcal{K}}, \mathcal{G})$, there hold ${\rm Im} \left(YB\right)={\rm Im} \Gamma$ and $\ker  \left(B^{\top}Y^{\top}\right)=\ker (\Gamma^{\top})$  if \cref{condition. 2} holds.
\end{proposition}
\begin{proof}
	The result follows immediately from the fact that the matrix $B$ and the incidence matrix $D$ can be derived from each other.
\end{proof}

For a network $(\mathcal{S},\mathcal{C}, \bar{\mathcal{R}}, \bar{\mathcal{K}}, \mathcal{G})$ that admits a positive detailed balanced equilibrium $(U^*,N^*)$ with $U^{*}>{N^*}^{\top}u(0)$,
we term $K_{\rm CR}(T) \in \mathbb{R}^{\frac{r_{\rm CR}}{2}\times \frac{r_{\rm CR}}{2}}$ as a diagonal matrix-valued function with 
\begin{eqnarray}
(K_{\rm CR}(T))_{j,j}&\triangleq&\left({T}/{T^*}\right)\exp\left\{-\frac{g^{\rm AS}_{2j-1}(T)}{A_vRT}+\frac{g^{\rm AS}_{2j-1}(T^*)}{A_vRT^*}\right\}
v_{2j-1}(U^*,N^*) \notag \\
&=& \left({T}/{T^*}\right)\exp\left\{-\frac{g^{\rm AS}_{2j}(T)}{A_vRT}+\frac{g^{\rm AS}_{2j}(T^*)}{A_vRT^*}\right\}
v_{2j}(U^*,N^*), \notag
\end{eqnarray}
where $T^*$ is the temperature at the state $(U^*,N^*)$, and the second equality follows immediately from the definition of the detailed balanced equilibrium and \cref{condition. 5}.
By \eqref{eq. transition state theory}, we can also write the diagonal element of $K_{\rm CR}(T)$ as 
\begin{eqnarray}
(K_{\rm CR}(T))_{j,j}&=&\tilde{k}_{2j-1}T\exp\left\{-\frac{g^{AS}_{2j-1}(T)}{A_vRT}\right\} \exp\left(y^{\top}_{\sigma_{2j-1}}\left(\Ln N^* +\frac{g(T^*)}{A_vRT^*}\right)\right) {\label{eq. KCR forward}} \notag \\
&=& \tilde{k}_{2j}T\exp\left\{-\frac{g^{AS}_{2j}(T)}{A_vRT}\right\} \exp\left(y^{\top}_{\pi_{2j-1}}\left(\Ln N^* +\frac{g(T^*)}{A_vRT^*}\right)\right) {\label{eq. backward}}, \notag
\end{eqnarray}
from which we can learn $K_{\rm CR}(T)$ is always positive definite for $T>0$.
By \eqref{eq. transition state theory} and the notation of $K_{\rm CR}(T)$, the rates of chemical reactions in $\bar{\mathcal{R}}_{\rm CR}$ can be expressed as
\begin{align}
v_{2j-1}(U,N)&= (K_{\rm CR}(T))_{j,j} \exp\left(y^{\top}_{\sigma_{2j-1}} \left( \Ln N+\frac{g(T)}{A_v R T}-\Ln N^*-\frac{g(T^*)}{A_v R T^*} \right) \right), \label{eq. forward reaction in CR} \\
v_{2j} (U,N) &= (K_{\rm CR}(T))_{j,j} \exp\left(y^{\top}_{\pi_{2j-1}} \left( \Ln N+\frac{g(T)}{A_v R T}-\Ln N^*-\frac{g(T^*)}{A_v R T^*} \right) \right) \label{eq. back reaction in CR}
\end{align}
for $j=1,\dots, r_{\rm CR}/2$, and, moreover,we can conclude 
\begin{align}
YD_{\rm CR} \cdot v_{CR}(U,N)
&= \sum_{j=1}^{{r_{\rm CR}}/{2}} YB_{\cdot j} \left(v_{2j-1}(U,N) -v_{2j}(U,N)\right)  \notag \\
&= -YB_{\rm CR} K_{\rm CR}(T)B_{\rm CR}^{\top} \Exp\left(Y \left( \frac{\mu}{A_v R T} - \frac{\mu^*}{A_v RT^*} \right) \right),  \label{eq. compact formula CR}
\end{align} 
where $\mu=\Ln N+\frac{g(T)}{A_v R T}$ is the chemical potential of the system (c.f. \eqref{eq. fundamental equation} and \cref{pro. fundamental equation}) and $\mu^*=\Ln N^*+\frac{g(T^*)}{A_v R T^*}$ is the chemical potential of the system at $(U^*,N^*)$.

Similarly, we term $K_{\rm IO}\in\mathbb{R}^{{\frac{r_{\rm CR}}{2}\times \frac{r_{\rm CR}}{2}}}$ as a constant diagonal matrix with
\begin{equation*}
\left(K_{\rm IO}\right)_{j,j}\triangleq v_{2j-1+r_{\rm CR}}(U^*, N^*) = v_{2j+r_{\rm CR}}(U^*, N^*),
\end{equation*} 
where the second equality follows from the definition of the detailed balanced equilibrium.
Provided with \cref{condition. 4} which suggests the kinetics of reactions in $\bar{\mathcal{R}}_{\rm IO}$ to follow mass-action laws only, we can write diagonal elements of $K_{\rm IO}$ alternatively by
\begin{equation*}
\left(K_{\rm IO}\right)_{j,j} = \tilde{k}_{2j-1+r_{\rm CR}} \exp \left( y^{\top}_{\sigma_{2j-1+r_{\rm CR}}} \Ln N^* \right)= \tilde{k}_{2j+r_{\rm CR}} \exp \left( y^{\top}_{\sigma_{2j+r_{\rm CR}}} \Ln N^* \right),
\end{equation*}
from which we can learn $K_{\rm IO}$ is also positive definite. 
By \eqref{eq. transition state theory}, \cref{condition. 4} and the notation of $K_{\rm IO}$, the rates of reactions in $\bar{\mathcal{R}}_{\rm IO}$ can be expressed as 
\begin{align}
v_{2j-1+r_{\rm CR}}(U,N)&= (K_{\rm IO}(T))_{j,j} \exp\left(y^{\top}_{\sigma_{2j-1+r_{\rm CR}}} \left( \Ln N-\Ln N^* \right) \right), \label{eq. forward reaction rate in IO} \\
v_{2j+r_{\rm CR}} (U,N) &= (K_{\rm IO}(T))_{j,j} \exp\left(y^{\top}_{\pi_{2j-1+r_{\rm CR}}} \left( \Ln N-\Ln N^*\right) \right)\label{eq. back reaction rate in IO}
\end{align}
for $j=1,\dots,r_{\rm IO}/2$, and, therefore, we can conclude 
\begin{align}
YD_{\rm IO} \cdot v_{IO}(U,N)
&= \sum_{j=1}^{{r_{\rm CR}}/{2}} YB_{\cdot j+r_{\rm CR}/2} \left(v_{2j-1+r_{\rm CR}}(U,N) -v_{2j+r_{\rm CR}}(U,N)\right)  \notag \\
&= -YB_{\rm IO} K_{\rm IO}B_{\rm IO}^{\top} \Exp\left(Y\left( \Ln N-\Ln N^* \right)  \right)  \label{eq. compact formula IO}.
\end{align} 

Also, note that the reaction in $\bar{\mathcal{R}}_{\rm HE}$, if any, points from a zero complex to itself, so $Y D_{\rm HE}$ is always zero.
Moreover, provided with \cref{condition. 4}, we can write the rate of the reaction in $\bar{\mathcal{R}}_{\rm HE}$ (if any) as
\begin{equation}\label{eq. compact formula HE}
v_{r_{\rm CR}+r_{\rm IO}+1}(U,N)= \tilde k_{r_{\rm CR}+r_{\rm IO}+1} \left(T_e-T\right)=\tilde k_{r_{\rm CR}+r_{\rm IO}+1} \left(T^*-T\right),
\end{equation} 
where the second equality follows from \eqref{eq. temperature at detailed balanced equilirbium}.

Finally, by \eqref{eq. dynamics of non-isothermal crn splited}, \eqref{eq. compact formula CR} \eqref{eq. compact formula IO}, and \eqref{eq. compact formula HE}, the  dynamics of a non-isothermal detailed balanced network $(\mathcal{S},\mathcal{C}, \bar{\mathcal{R}}, \bar{\mathcal{K}}, \mathcal{G})$ can be expressed  as
\begin{align}
\dot U =& \Delta \mathcal U_{\rm CR} v_{CR}(U,N)+\Delta \mathcal U_{\rm IO} v_{IO}(U,N) +  \tilde k_{r_{\rm CR}+r_{\rm IO}+1} \left(T^*-T\right)\chi(r_{\rm HE}\neq 0), \notag \\
\dot N =& -YB_{\rm CR} K_{\rm CR}(T)B_{\rm CR}^{\top} \Exp\left(Y \left( \frac{\mu}{A_v R T} - \frac{\mu^*}{A_v RT^*}  \right) \right)  \label{eq. compact formula} \\
& -YB_{\rm IO} K_{\rm IO}B_{\rm IO}^{\top} \Exp\left(Y\left( \Ln N-\Ln N^* \right)  \right), \notag 
\end{align}
if \cref{condition. 2}, \cref{condition. 4}, and  \cref{condition. 5} hold.
Compared with \eqref{eq. dynamics of non-isothermal CRNs} and \eqref{eq. dynamics of non-isothermal crn splited}, the dynamic equation \eqref{eq. compact formula} depicts the rate and the driving force of each reaction alternatively by thermodynamic quantities, $\mu(U,N)$ and $\Ln N $ $\left(=\frac{\mu(U,N)-{g(T)}}{A_vRT}\right)$. 
Although the formula \eqref{eq. dynamics of non-isothermal crn splited} seems complicated in expression, it is compact with respect to information, as it exhibits at the same time the network topology, $B$, and the thermodynamic information, $\mu$.
In the next subsection, some basic properties of detailed balanced networks, mainly the stability and the detailed balancing of each equilibrium, are carried out based on this formula. 

In the isothermal case where $T\equiv T^*=T_e$, the dynamics \eqref{eq. compact formula} becomes 
\begin{align}
\dot U =& -\left(u(T_e)\right)^{\top}YB K B^{\top} \Exp\left(Y \left(\Ln N-\Ln N^*\right)\right) \notag\\
\dot N =& -YB K B^{\top} \Exp\left(Y \left(\Ln N-\Ln N^*\right)\right)  \label{eq. compact formula of arjan's} 
\end{align}
where 
\begin{equation*}
K\triangleq \left(
\begin{array}{cc}
K_{\rm CR} (T_e) & \\ & K_{IO} 
\end{array}
\right).
\end{equation*}
The dynamic equation \eqref{eq. compact formula of arjan's} is exactly the one provided in \cite{van2013mathematical} for isothermal CRNs; therefore, our dynamic equations \eqref{eq. compact formula} can be viewed as an extension of that formula (in \cite{van2013mathematical}) to non-isothermal cases.
Compared with the corresponding formula in \cite{van2013mathematical} where $K$ is always constant (c.f. \eqref{eq. compact formula of arjan's}), the dynamic equation \eqref{eq. compact formula} has a diagonal matrix-valued function $K_{\rm CR}(T)$ which can vary with respect to the temperature due to the transition state theory and exhibits the physical nature of non-isothermal systems. 

In another special case where the system is isolated, i.e. $r_{\rm IO}=r_{\rm HE}=0$, the dynamics \eqref{eq. compact formula} become 
\begin{align}
\dot U =& 0, \notag \\
\dot N =& -YB_{\rm CR} K_{\rm CR}(T)B_{\rm CR}^{\top} \Exp\left(Y \left( \frac{\mu}{A_v R T} - \frac{\mu^*}{A_v RT^*}  \right) \right),  \notag
\end{align}
which is exactly the formula provided in \cite{wang2018port}.
This formula can also be written in the form of a port-Hamiltonian system (see \cite[(21)-(23)]{wang2018port}). 
Compared with it, our formula \eqref{eq. compact formula} has additional terms, ``$-YB_{\rm IO} K_{\rm IO}B_{\rm IO}^{\top} \Exp\left(Y\left( \Ln N-\Ln N^* \right)  \right)$" and ``$\tilde k_{r_{\rm CR}+r_{\rm IO}+1} \left(T^*-T\right)\chi(r_{\rm HE}\neq 0)$", depicting the effect of boundary fluxes and heat exchanges, and, therefore, is able to cover a broader class of systems under the same framework.

\subsection{Stability and detailed balancing of each equilibrium}
We first show the stability of a detailed balanced network through Lyapunov's second method.
We term  
\begin{align}{\label{eq. availability function of entropy}}
&S_{\mathcal A}(U, N) \\
& \triangleq -S(U,N) + \frac{\partial S (U^{*},N^{*})}{\partial U} \left(U-U^*\right)+ \frac{\partial S (U^{*},N^{*})}{\partial N} \left(N-N^*\right)+S(U^{*},N^{*}) \notag
\end{align}
as the availability function of the negative entropy with respect to a positive detailed balanced equilibrium $(U^{*},N^{*})$, which depicts the difference between the entropy and the supporting hyperplane at $(U^{*},N^{*})$ (c.f. \cite{ydstie1997process,alonso2001stabilization}).
Provided with the condition \eqref{eq. at least linear growth of the Hamiltonian}, the function $S_{\mathcal A}(U, N)$ is strictly positive at any positive state other than $(U^{*},N^{*})$ due to the strict concavity of the entropy (see \cref{pro. strict concavity}).
The positive definiteness, together with thermodynamics underlying this function, suggests $S_{\mathcal A}(U, N)$ to be a good Lyapunov function candidate for investigating thermodynamic processes \cite{ydstie1997process,ydstie2002passivity,alonso2001stabilization}. 
In the isothermal case where $T\equiv T^{*}$, we can rewrite this availability function, by \eqref{eq. small S help}, as
\begin{equation}{\label{eq. connection of the availability function and pseudo helmoholtz free energy}}
S_{\mathcal A}(U, N)
=\frac{G(T^*,N)-\frac{\partial G(U^{*},N^*)}{\partial N} (N-N^{\dagger})-G(T^{*},N^{*})}{T^{*}}
= \frac{1}{A_v R}G_{\mathcal{A}}(N)
\end{equation}
which implies that two functions, $S_{\mathcal A}(U, N)$ and $G_{\mathcal{A}}(N)$, are closely connected and that $S_{\mathcal A}(U, N)$ can automatically serve as a Lyapunov function for detailed balanced networks whose temperature is fixed. 
Moreover, $S_{\mathcal A}(U, N)$ is also proven to be a Lyapunov function for isolated non-isothermal detailed balanced networks \cite{wang2018port}.
In the following discussion, we further illustrate that the availability function \eqref{eq. availability function of entropy} serves as a Lyapunov function for every non-isothermal detailed balanced network.

\begin{theorem}{\label{thm. stability}}
	Let $(\mathcal{S},\mathcal{C}, \bar{\mathcal{R}}, \bar{\mathcal{K}}, \mathcal{G})$ be a non-isothermal detailed balanced network admitting a positive detailed balanced equilibrium $(U^*,N^*)$, where $U^*=(N^*)^{\top}u(T_e)$ if $\Delta \mathcal{U}_{\rm CR}(U,N )\neq \mathbbold{0}^{r_{\rm CR}}$. 
	Provided with \eqref{eq. at least linear growth of the Hamiltonian}, \cref{condition. 2}, \cref{condition. 3}, \cref{condition. 4}, and \cref{condition. 5}, it follows that  
	\begin{align*}
	\dot{S}_{\mathcal A}(U, N) \leq 0, \quad \forall (U,N) \in 
	\left\{\begin{array}{lr}
	\mathbb{R}^{n+1}_{> 0} \bigcap \{U> N^{\top}u(0)\} & \text{if } \Delta \mathcal{U}_{\rm CR}\neq \mathbbold{0}^{r_{\rm CR}} \\
	\mathbb{R}^{n+1}_{> 0} \bigcap \{U= N^{\top}u(T_e)\} & \text{if } \Delta \mathcal{U}_{\rm CR}= \mathbbold{0}^{r_{\rm CR}} \\
	\end{array}
	\right.
	\end{align*}
	where the equality holds if and only if 
	$ \nabla{S}_{\mathcal A}(U, N)
	=\left(\frac{1}{T^*}-\frac{1}{T},\frac{\mu}{T}-\frac{\mu^*}{T^*}\right)^{\top} 
	\in \ker (\tilde{ \Gamma}^{\top})$.
	Therefore, $S_{A}(U, N)$ is a Lyapunov function for the system \eqref{eq. compact formula} rendering the equilibrium $(U^*,N^*)$ to be stable.
\end{theorem}

\begin{proof}
	The proof is shown in \cref{section proof of stability}.
	In short, the result follows from the fact that $B_{\rm CR} K_{\rm CR}(T)B_{\rm CR}^{\top}$ and $YB_{\rm IO} K_{\rm IO}B_{\rm IO}^{\top}$ are both Laplacian matrices.
\end{proof}

The above theorem shows the stability result for non-isothermal detailed balanced networks.
Compare to the corresponding result for isothermal networks (the first result of \cref{thm. isothermal detailed balanced}),
both results indicate the availability function ($G_{\mathcal A}(\cdot)$ or $S_{\mathcal A}(\cdot )$) to serve as Lyapunov functions, whose gradient is orthogonal to the stoichiometric(-like) subspace at each non-dissipative state.
Though more conditions are required in the above theorem than in \cref{thm. isothermal detailed balanced}, these conditions are easy to verify in practical systems, and, therefore, the two stability results for respectively isothermal detailed balanced networks and non-isothermal ones are quite similar in form. 
Moreover, in the isothermal case where $T\equiv T_e=T^*$, the availability function $S_{\mathcal A}(U,N)$  is identical to $\frac{G_{\mathcal A } (N)}{A_{v}R}$ (c.f. \eqref{eq. connection of the availability function and pseudo helmoholtz free energy}), and the above theorem becomes exactly the first result in \cref{thm. isothermal detailed balanced}.
Consequently, we can view \cref{thm. stability} as a non-isothermal analog to the first result of \cref{thm. isothermal detailed balanced}.

By the above theorem, we can further arrive at the detailed balancing of each positive equilibrium, which extends the second result of \cref{thm. isothermal detailed balanced}.

\begin{corollary}{\label{cor. detailed balancing of each equilibrium}}
	Let $(\mathcal{S},\mathcal{C}, \bar{\mathcal{R}}, \bar{\mathcal{K}}, \mathcal{G})$ be a non-isothermal detailed balanced network with a positive detailed balanced equilibrium $(U^*,N^*)$, where $U^*=(N^*)^{\top}u(T_e)$ if $\Delta \mathcal{U}_{\rm CR}\neq \mathbbold{0}^{r_{\rm CR}}$. 
	If all conditions in \cref{thm. stability} are satisfied, 
	then the following statements about a positive state $(U^{**},N^{**})$, where $U^{**}=(N^{**})^{\top}u(T_e)$ if $\Delta \mathcal{U}_{\rm CR}\neq \mathbbold{0}^{r_{\rm CR}}$, are equivalent. 
	\begin{enumerate}
		\item $(U^{**},N^{**})$ is an equilibrium.
		\item $
		\nabla{S}_{\mathcal A}(U^{**}, N^{**})
		\in \ker (\tilde{ \Gamma}^{\top})$.
		\item $(U^{**},N^{**})$ is a detailed balanced equilibrium.
	\end{enumerate}
\end{corollary}

\begin{proof}
	``1$\Rightarrow$2" follows immediately from \cref{thm. stability}, and ``3$\Rightarrow$1" follows from the definition.
	Therefore, we only need to show ``2$\Rightarrow$3".
	
	By the second statement in the above, the definition of $\tilde \Gamma$, and \cref{prop. stoichiometric compatibility class}, we can arrive at 
	\begin{equation}\label{eq. proof detailed balancing 1}
	{\mu^{**}}/{T^{**}}-{\mu^*}/{T^*} \in \ker(B^{\top}Y^{\top}).
	\end{equation}
	where $T^{**}$ and $\mu^{**}$ are the temperature and chemical potentials at $(U^{**}, N^{**})$.
	
	\textit{Balance of reaction rates:}
	For each pair of forward and backward reactions in $\bar{\mathcal{R}}_{\rm CR}$, the relation \eqref{eq. proof detailed balancing 1} suggest $(y_{\sigma_{2j-1}}-y_{\pi_{2j-1}})^{\top}\left(\frac{\mu^{**}}{T^{**}}-\frac{\mu^*}{T^*} \right)=0$ for $j=1,\dots, r_{\rm CR}/2$, and, therefore, 
	\begin{align*}
	v_{2j-1}(U^{**},N^{**})=& (K_{\rm CR}(T))_{j,j} \exp\left(y^{\top}_{\sigma_{2j-1}} \left( \frac{\mu^{**}}{T^{**}}-\frac{\mu^*}{T^*} \right) \right)  \\
	=&  (K_{\rm CR}(T))_{j,j} \exp\left(y^{\top}_{\pi_{2j-1}} \left( \frac{\mu^{**}}{T^{**}}-\frac{\mu^*}{T^*}\right) \right) \\
	=& v_{2j}(U^{**},N^{**}), & j=1,\dots, r_{\rm CR}/2
	\end{align*}
	where the first and last equality follows from \eqref{eq. forward reaction in CR} and \eqref{eq. back reaction in CR} respectively. Thus, we show balance of reaction rates for each pair of forward and backward reactions in  $\bar{\mathcal{R}}_{\rm CR}$.
	For each pair of reversible reactions in $\bar{\mathcal{R}}_{\rm IO}$ (if any), we can similarly prove the result by \eqref{eq. proof detailed balancing 1}, \eqref{eq. temperature at detailed balanced equilirbium}, \eqref{eq. forward reaction rate in IO}, and \eqref{eq. back reaction rate in IO}.
	For the reaction in $\bar{\mathcal{R}}_{\rm HE}$, the result is trivial, as its backward reaction is itself.
	
	\textit{Balance of energy changes:} For each pair of forward and backward reactions in $\bar{\mathcal{R}}_{\rm CR}$, balance of energy changes is trivial, as $\Delta \mathcal{U}_{j}$ ($j=1,\dots, r_{\rm CR}$) are constants. 
	For each pair in $\bar{\mathcal{R}}_{\rm IO}$ and $\bar{\mathcal{R}}_{\rm HE}$ (if any), balance of energy changes follows immediately from $T^{**}=T^*=T_e$ (c.f. \eqref{eq. temperature at detailed balanced equilirbium}) and the restrictions of $\Delta \mathcal{U}_{j}$ in $\bar{\mathcal{R}}_{\rm IO}$ and $\bar{\mathcal{R}}_{\rm HE}$. 
	
	Given the analysis above, we show ``2$\Rightarrow$3" which proves the result.
\end{proof}

\section{Asymptotic stability of detailed balanced network}{\label{section asymptotic stability}}
In this section, we show the asymptotic stability of a detailed balanced network to the unique equilibrium in each invariant manifold, which extend the third result of \cref{thm. isothermal detailed balanced}.
From \cref{thm. stability} and \cref{cor. detailed balancing of each equilibrium}, we can learn that the Lyapunov function $S_{\mathcal A }(U,N)$ dissipates at any positive state except the detailed balanced equilibrium. 
Therefore, by the second Lyapunov's method, showing the asymptotic stability of the system is equivalent to proving the existence and uniqueness of the detailed balanced equilibrium in each positive stoichiometric-like compatibility class.

\subsection{Legendre transformation and the proof scheme}

Geometrically, the fact that the gradient of the convex function $S_{\mathcal A }(U,N)$ is orthogonal to the stoichiometric-like subspace at a positive detailed balanced equilibrium (see \cref{cor. detailed balancing of each equilibrium}) suggests the equilibrium to be  a  minimum point of the Lyapunov function in the corresponding positive stoichiometric-like compatibility class.
Therefore, a straightforward idea to investigate asymptotic stability is to check whether the convex function $S_{\mathcal A }(U,N)$ has a unique minimum point in an invariant set $\mathcal{PS}(U^o,N^{o})$.
However, this idea is not easy to realize, because $S_{\mathcal A }(U,N)$ takes finite values at finite  boundary points (c.f. \eqref{eq. entropy function whole system}), and we cannot tell if $S_{\mathcal A }(U,N)$ has a minimum value point in such a case.


Instead, we study this problem by investigating the Legendre transformation of the Lyapunov function $S_{\mathcal A }(U,N)$.
The Legendre transformation switches the positions of a function's independent variables and derivatives, as a consequence of which it can preserve geometric relations between these variables while possibly generate a function that is unbounded at any boundary.
In what follows, we flesh out the details of the scheme of our analysis.

First, we denote the Legendre transformation of $S_{\mathcal A }(U,N)$ by
\begin{equation}{\label{eq. legendre transformation}}
\mathcal{L}(\beta, \gamma) \triangleq -S_{\mathcal A }(U,N) + \beta U + \gamma^{\top} N,
\end{equation}
where $\beta=\frac{\partial S_{\mathcal A }(U,N)}{\partial U}=\frac{1}{T^*}-\frac{1}{T}$,  $\gamma=\left(\frac{\partial S_{\mathcal A }(U,N)}{\partial N}\right)^{\top}=\frac{\mu}{T}-\frac{\mu^*}{T^*}$, and the definition domain is $\mathfrak{D}\triangleq(-\infty,\frac{1}{T^*})\otimes \mathbb{R}^{n}$.
By denoting $T^o$ and $\mu^o$ the temperature and chemical potentials at $(U^o,N^o)$, we further modify the Legendre function by shifting it at $(\beta^o,\gamma^o)\triangleq(\frac{1}{T^*}-\frac{1}{T^o},\frac{\mu^o}{T^o}-\frac{\mu^*}{T^*} )$ and obtain
\begin{equation*}
\mathcal{L}_{\mathcal{A}}(\beta,\gamma)= \mathcal{L}(\beta,\gamma)-(\beta-\beta^o)U^o-(\gamma-\gamma^o)^{\top}N^o-\mathcal{L}(\beta^o,\gamma^o).
\end{equation*}
Here, we term the modified function  as $\mathcal{L}_{\mathcal{A}}(\cdot)$, because it is also an availability function of the Legendre transformation with respect to $(\beta^o,\gamma^o)$.
Note that functions' convexities are preserved under both the Legendre transformation and the availability function transformation.
Therefore, $L(\beta,\gamma)$ and $L_{\mathcal A}(\beta,\gamma)$ are both strictly convex.  
In an invariant set $\mathcal{PS}(U^o,N^o)$ where $U^o=(N^o)^{\top}u(T_e)$ if $\Delta \mathcal{U}_{\rm CR}\neq \mathbbold{0}^{r_{\rm CR}}$, \cref{cor. detailed balancing of each equilibrium} suggests a positive state $(U^{**},N^{**})$ to be detailed balanced if and only if
\begin{equation*}
\left(
\begin{array}{c}
\beta^{**} \\ \gamma^{**}
\end{array}
\right)
\triangleq
\left(
\begin{array}{c}
\frac{1}{T^*}-\frac{1}{T^{**}} \\ \frac{\mu^{**}}{T^{**}}-\frac{\mu^*}{T^*}
\end{array}
\right)
\in 
\mathfrak{D} \cap \ker (\tilde{\Gamma}^{\top}) 
= \mathfrak{D} \cap  ({\rm Im}\tilde{\Gamma})^{\perp} 
\end{equation*}
and 
\begin{equation*}
\nabla \mathcal{L}_{\mathcal{A}}(\beta^{**},\gamma^{**})
=
\left(
\begin{array}{c}
U^{**}-U^o \\ N^{**}-N^{o}
\end{array}
\right) \in {\rm Im} \tilde{ \Gamma},
\end{equation*}
where $T^{**}$ and $\mu^{**}$ are respectively the temperature and chemical potentials at the state $(U^{**},N^{**})$.
From the above relations, we can conclude by the convexity of $\mathcal{L}_{\mathcal A}(\cdot)$ that $(\beta^{**},\gamma^{**})$ is a minimum point of $\mathcal{L}_{\mathcal A}(\beta,\gamma)$   in the region $\mathfrak{D} \cap  ({\rm Im}\tilde{\Gamma})^{\perp} $ if and only if $(U^{**},N^{**})$ is a detailed balanced equilibrium. Thus, showing the existence and uniqueness of a detailed balanced equilibrium is equivalent to proving the existence and uniqueness of a minimum point of $\mathcal{L}_{\mathcal A}(\beta,\gamma)$ in the region $\mathfrak{D} \cap  ({\rm Im}\tilde{\Gamma})^{\perp} $.
A straightforward way to verify it is to show $\mathcal{L}_{\mathcal A}(\beta,\gamma)$ to go to infinity at boundary points. 
Thanks to thermodynamic knowledge, especially the low temperature and high temperature limit of thermodynamic quantities, such analysis is feasible (see \cref{pro. unboundedness of LA}).

In the isothermal case, the existence and uniqueness of a detailed balanced equilibrium in a positive stoichiometric compatibility class $(N^o + {\rm Im}\Gamma)\cap \mathbb{R}^n_{>0}$ is shown by investigating an auxiliary function (c.f. \cite{feinberg1995existence})
\begin{equation*}
\psi(\gamma) =(N^*)^{\top} \Exp(\gamma)-\gamma^{\top} N^o.
\end{equation*}
Note that the Legendre transformation of the pseudo-Helmholtz free energy is 
\begin{equation*}
\mathcal{L}_{G_\mathcal{A}}(\gamma) \triangleq -G_{\mathcal A}(N)+\gamma^{\top} N = (N^*)^{\top} \Exp (\gamma) -N^*,
\end{equation*}
where $\gamma=\nabla G_{\mathcal A}(N)$,
and its shifted function at $\gamma=\Ln N^o-\Ln N^*$ is given by
\begin{align*}
\mathcal{L}_{\mathcal A, G_\mathcal{A}}(\gamma) 
&\triangleq \mathcal{L}_{G_\mathcal{A}}(\gamma)- \gamma^{\top}\nabla \mathcal{L}_{G_\mathcal{A}}(\gamma^o) -\mathcal{L}_{G_\mathcal{A}}(\gamma^o) \\
&= (N^*)^{\top} \Exp (\gamma) - \gamma^{\top}N^o+
(N^*)^{\top} \Exp (\gamma^o),
\end{align*} 
which only differs from the auxiliary function $\psi(\gamma)$ in a constant $(N^*)^{\top} \Exp (\gamma^o)$. 
Therefore, though not pointed out explicitly in \cite{feinberg1995existence}, the Legendre transformation also governs the asymptotic stability analysis of isothermal detailed balanced networks. 

\subsection{Asymptotic stability}\label{as_gao}
Following the proof scheme introduced in the previous subsection, we can arrive at the main results in this paper.

First, we point out that the variables $T$ and $N$ can be expressed as
\begin{align}
T&= \left(\frac{1}{T^*}-\beta \right)^{-1}, \label{eq. express of temperature by beta gamma}\\
N_i&=\exp\left(\frac{\gamma+\frac{\mu^*}{T^*}-\frac{g_i(T)}{T}}{A_v R} \right)
=Z_i(T)\exp\left(\frac{\gamma+\frac{\mu^*}{T^*}}{A_v R} \right)
, \quad i=1,\dots, n. {\label{eq. express of mass amount by beta gamma}}
\end{align}
Given the above expressions, we further use \cref{condition. 1} to show the unboundedness of $\mathcal{L}_{A}(\beta,\gamma)$ at any boundary. 

\begin{proposition}{\label{pro. unboundedness of LA}}
	Provided with \eqref{eq. at least linear growth of the Hamiltonian}, \cref{condition. 1}, and $U^o>(N^o)^{\top}u(0)$, the function $\mathcal{L}_{A}(\beta,\gamma)$ satisfy 
	\begin{align}
	\lim_{(\beta,\gamma)\to ({\beta^{b}}^{-},\gamma^{b})} \mathcal{L}_{\mathcal{A}} (\beta,\gamma)&= +\infty, &  \forall\left(\beta^{b},\gamma^{b}\right) \in \left\{\frac{1}{T^*}\right\} \otimes \mathbb{R}^{n}, \label{eq. unboundedness of LA at finite boundary}\\
	\lim_{\theta \to +\infty} \mathcal{L}_{\mathcal{A}}\left(\theta \bar\beta,\theta \bar \gamma\right)&= + \infty, & \forall (\bar \beta, \bar \gamma) \in \bigg((-\infty,0) \otimes \mathbb{R}^{n} \bigg) / (0,\mathbbold{0}_{n}).
	\label{eq. unboundedness of LA at infinite boundary} 
	\end{align}
\end{proposition}
\begin{proof}
	The proof may be found in \mbox{\cref{section proof of asymptotic stability}}.
\end{proof}

Then we prove the existence and uniqueness of a minimum point in $\mathfrak{D} \cap  \ker (\tilde{\Gamma}^{\top})$. 

\begin{proposition}{\label{pro. convexity and closeness of SC}}
	Provided with \eqref{eq. at least linear growth of the Hamiltonian}, \cref{condition. 1}, and $U^o>(N^o)^{\top}u(0)$, 
	the set $$\mathfrak{S}_{c} \triangleq \{(\beta,\gamma) \in \mathfrak{D} \cap  \ker (\tilde{\Gamma}^{\top})~|~ \mathcal{L}_{A}(\beta,\gamma)\leq \mathcal{L}_{A}(0,\mathbbold{0}_{n}) \}$$ is non-empty, convex, and closed.
\end{proposition}

\begin{proof}
	See details in \mbox{\cref{section proof of asymptotic stability}}.
\end{proof}

\begin{proposition}{\label{pro. compactness of SC}}
	Provided with \eqref{eq. at least linear growth of the Hamiltonian}, \cref{condition. 1} (in \cref{section thermodynamics}), and $U^o>(N^o)^{\top}u(0)$, 
	the set $\mathfrak{S}_{c}$ is compact.
\end{proposition}

\begin{proof}
	See details in \mbox{\cref{section proof of asymptotic stability}}.
\end{proof}

\begin{proposition}{\label{pro. minimum point of LA}}
	Provided with \eqref{eq. at least linear growth of the Hamiltonian}, \cref{condition. 1}, and $U^o>(N^o)^{\top}u(0)$, 
	the shifted Legendre function $\mathcal{L}_{A}(\beta,\gamma)$ has a unique minimum in $\mathfrak{D} \cap  \ker (\tilde{\Gamma}^{\top})$.
\end{proposition}

\begin{proof}
	See details in \mbox{\cref{section proof of asymptotic stability}}.
\end{proof}

Finally, we show the existence and uniqueness of a detailed balanced equilibrium in each positive stoichiometric compatibility class and the asymptotic stability of the equilibrium. 

\begin{theorem}{\label{thm. asymptotic stability}}
	Let $(\mathcal{S},\mathcal{C}, \bar{\mathcal{R}}, \bar{\mathcal{K}}, \mathcal{G})$ be a non-isothermal detailed balanced network with a positive detailed balanced equilibrium $(U^*,N^*)$, where $U^*=(N^*)^{\top}u(T_e)$ if $\Delta \mathcal{U}_{\rm CR}\neq \mathbbold{0}^{r_{\rm CR}}$. 
	Provided with \eqref{eq. at least linear growth of the Hamiltonian} and Conditions 1--5, there exists a unique positive detailed balanced equilibrium in each positive stoichiometric-like compatibility class $\mathcal{PS}(U^o,N^o)$, where 
	\begin{itemize}
		\item $U^o>(N^o)^{\top}u(0)$, 
		\item  $U^o=(N^o)^{\top}u(T_e)$ if $\Delta \mathcal{U}_{\rm CR}\neq \mathbbold{0}^{r_{\rm CR}}$.
	\end{itemize}
	Moreover, this equilibrium is locally asymptotically stable.
\end{theorem}
\begin{proof}
	The detailed proof is shown in \cref{section proof of asymptotic stability}.
\end{proof}

\subsection{Some future remarks}

\cref{thm. asymptotic stability} extends the third result in \cref{thm. isothermal detailed balanced}, and, thus, non-isothermal detailed balanced networks inherit all results in \cref{thm. isothermal detailed balanced}: the dissipativeness (\cref{thm. stability}), the detailed balancing of each equilibrium (\cref{cor. detailed balancing of each equilibrium}), the existence and uniqueness of the equilibrium (\cref{thm. asymptotic stability}), and the asymptotic stability (\cref{thm. asymptotic stability}). 

From the previous discussions, we can observe \cref{condition. 1} (in \cref{section thermodynamics}), which suggests the unboundedness of each partition function $Z_i(\cdot)$, plays an important role in showing \cref{thm. asymptotic stability}  by guaranteeing the unboundedness of the modified Legendre transformation $\mathcal{L}_{A}(\beta,\gamma)$ at any boundary.
Note that this condition is absent in any result prior to this section; 
therefore, in this paper, it serves specifically for the asymptotic stability problem. 
Moreover, one can also observe that all propositions concerning $\mathcal{L}_{A}(\beta,\gamma)$ in \cref{section proof of asymptotic stability} are independent of the specific structure of a non-isothermal CRN and purely thermodynamic knowledge, which implies that thermodynamics can help researchers gain insights into CRNT. 

Also, we need to emphasize that the asymptotic stability shown in \cref{thm. asymptotic stability} is a local result rather than a global one. 
It is caused by the failure of the second Lyapunov method  to preclude boundary $\omega$-limit points provided that the Lyapunov function is finite at finite boundaries and can be non-dissipative on them.
It still remains an open problem whether a non-isothermal detailed balanced network is globally asymptotically stable or not.
To further investigate this problem, researchers are required to study the persistence of non-isothermal detailed balanced networks, i.e., to check whether or not boundary $\omega$-limit points exist.
Some methodologies \cite{gopalkrishnan2014geometric,craciun2015toric,anderson2008global,anderson2011proof} that succeed in investigating the persistence of isothermal CRNs can be helpful to this problem.

Notwithstanding the above limitations, our current modeling method still has enough advantages compared to the existing one, the port-Hamiltonian approach \mbox{\cite{wang2016irreversible,wang2018port}}. The latter models non-isothermal detailed balanced chemical reaction systems in the form of port-Hamiltonian systems (see \mbox{\cite[(38)]{wang2018port}}) and captures stability through the passivity structure of the Hamiltonian systems, but fails to discuss the asymptotic stability of a detailed balanced equilibrium. However, the current approach yields deeper results in understanding the dynamics of non-isothermal reaction network systems and about the stable behavior of detailed balanced equilibria. As stated in \mbox{\cref{rem_3.2}}, it utilizes the transition state theory to characterize the kinetics that could provide a much clearer interpretation to reaction kinetics and a close connection between the kinetics and thermodynamic properties.
Based on these, a clear statement about the invariant manifold and the geometry of equilibrium sets may be obtained, which together with thermodynamic knowledge, particularly with the high temperature and low temperature limit, can serve for reaching the asymptotic stability of each detailed balanced equilibrium.
In addition, the current method could model three parts of an open reaction network system, i.e., non-isothermal chemical reaction networks, boundary inflow and outflow, and heat exchange, in the same framework, which also greatly benefit the analysis of dynamical behaviors. However, the work of Wang et. al. \mbox{\cite{wang2018port}} cannot model these three parts simultaneously using the same approach. This fact also suggests that our current modeling approach is suitable for describing a broader class of reaction systems. Naturally, it needs to be pointed out that in the special case of isolated reaction network systems the mentioned two modeling methods yield the same dynamic equations.

\section{Conclusion}{\label{section discussion}}
In this paper, we provide a graphic formulation for modeling non-isothermal chemical reaction systems based on the classical CRNT and apply it to analyzing dynamic properties of detailed balanced network systems.
To model thermal effects, we first extend the isothermal CRN by adding two parameters to each (reaction) edge depicting respectively the instantaneous energy change and the transition state theory.
The newly established networks are termed as non-isothermal CRNs and shown to be efficient in modeling a broad class of non-isothermal chemical reaction systems.
Moreover, we introduce detailed balanced networks for non-isothermal CRNs and provide them with a compact dynamic formula that exhibits at the same time the network topology and thermodynamic information.
With this compact formula, the Legendre transformation, and some mild conditions, we show non-isothermal detailed balanced network systems to admit some fundamental properties,
\begin{enumerate}
	\item the dissipativeness with respect to $S_{\mathcal A}(U,N)$ (\cref{thm. stability}), 
	\item the detailed balancing of each equilibrium (\cref{cor. detailed balancing of each equilibrium}),
	\item the existence, uniqueness, and asymptotic stability of the detailed balanced equilibrium (\cref{thm. asymptotic stability}),
\end{enumerate}
which are well consistent with results in isothermal detailed balanced networks.
In contrast with the corresponding results for isothermal cases, the above properties require more conditions for non-isothermal detailed balanced networks.
Specifically, \eqref{eq. at least linear growth of the Hamiltonian} ensures the well posedness of thermodynamic quantities, \cref{condition. 2}, \cref{condition. 3} and \cref{condition. 4} enable the non-isothermal CRNs to model a broad class of practical reaction systems, \cref{condition. 5} helps to construct a compacted formula for non-isothermal detailed balanced networks, and \cref{condition. 1} is used to prove the asymptotic stability by guaranteeing the modified Legendre transformation to be unbounded at any boundary.
In a practical system, these conditions are usually easy to verify and, therefore, not too restrictive.
In general, the analysis and results of this work, especially thermodynamic interpretations, provide insights into the research of non-isothermal chemical reaction systems. 

There are numerous topics that we are pursuing, related to non-isothermal CRNT.
First, for non-isothermal detailed balanced networks, it remains an open problem to extend the established locally asymptotic stability to a global result.
Some methods that succeed in investigating the persistence of isothermal CRNs, such as strongly endotactic networks \cite{gopalkrishnan2014geometric}, the tier structure \cite{anderson2018tier,anderson2011proof}, semi-locking sets \cite{anderson2008global}, and the toric differential inclusion \cite{gopalkrishnan2014geometric,craciun2015toric} can be helpful to this problem.
Second, we will also investigate the connection between the network topology and unique/multi-stationarity, which can be applied to chemical engineering designs to preclude undesired stationary states. 
We suspect that the deficiency theory can still work for the equilibrium analysis of non-isothermal CRNs after some necessary modifications.
Third, since oscillations are common phenomena in chemical engineering, it is also interesting to investigate what causes oscillations in non-isothermal CRNs.
Forth, we are also exploring the stochastic behaviors of non-isothermal CRNs when systems' scales are small and trying to extract thermodynamic knowledge from them.
The recent work \cite{fang2019thermodynamic} is a good starting point for this problem.
Finally, though most relevant topics of non-isothermal CRNs are in the field of chemical engineering, there are still some applications in biological studies. A most direct application is the network used by Wang et. al.\mbox{\cite{wang2018port}} that takes place in a very common protein synthesis circuit in the cell of \textit{E. Coli} in the large intestine of human beings. The following three reactions are included in this network: 

LuxR+AHL$\rightleftharpoons$LuxR-AHL,

2(LuxR-AHL)$\rightleftharpoons$(LuxR-AHL)$_2$,

(LuxR-AHL)$_2$+DNA$\rightleftharpoons$ DNA-(LuxR-AHL)$_2$.

\noindent More details about the network may be found in that paper.
Also, it can be used to design control laws for a cell culturing system, where the temperature needs to be accurately regulated.

\appendix

\section{Notations and terminologies}{\label{section notations}}
Here we introduce some notations and terminologies used in this paper.

\noindent{\textbf{Mathematical Notation:}}\\
\rule[1ex]{\columnwidth}{0.8pt}
\begin{description}
	\item[\hspace{+0.8em}{$\mathbb{R}^n,\mathbb{R}^n_{\geq 0},\mathbb{R}^n_{\textgreater 0}$}]: $n$-dimensional real space, nonnegative and positive real space, respectively.
	\item[\hspace{+0.8em}{$x^{v_{\cdot i}}$}]: $x^{v_{\cdot i}}=\prod_{j=1}^{d}x_{j}^{v_{ji}}$, where $x,v_{\cdot i}\in\mathbb{R}^{d}$ and $0^{0}$ is defined to be $1$.
	\item[\hspace{+0.8em}{$\mathrm{Exp}(x)$}]: $\Exp(x)=\left(\exp(x_{1}),\exp(x_{2}),\cdots,\exp({x}_{d})\right)^{\top}$ where $x\in\mathbb{R}^{d}_{>{0}}$.
	\item[\hspace{+0.8em}{$\mathrm{Ln}(x)$}]: $\Ln(x)=\left(\ln{x_{1}},\ln{x_{2}},\cdots,\ln{{x}_{d}}\right)^{\top}$, where $x\in\mathbb{R}^{d}_{>{0}}$.
	\item[\hspace{+0.8em}{$\otimes$}]: Cartesian product.
	\item[\hspace{+0.8em}{$\mathcal{I}_{n}$}]: {an $n\times n$ identity matrix.} 
	\item[\hspace{+0.8em}{$\mathbbold{0}_{n}$}]: {an $n$-dimensional vector with every entry to be zero.} 
	\item[\hspace{+0.8em}{$\mathbbold{1}_{n}$}]: {an $n$-dimensional vector with every entry to be one.}
	\item[\hspace{+0.8em}{$\delta_{i}$}]: {an $n$-dimensional vector with the $i$-th component being 1 and the others zeros.}
	\item[\hspace{+0.8em}{$\chi(\cdot)$}]: {the indicator function.}
\end{description}
\rule[1ex]{\columnwidth}{0.8pt}

\noindent{\textbf{Physics terminologies:}}\\
\rule[1ex]{\columnwidth}{0.8pt}
\begin{description}
	\item[Activated state:] the state corresponding to the highest potential energy along this reaction coordinate.
	\item[Callen's first postulate:] a well mixed system can be fully described by internal energy, system volume, and mass amounts of involved substances.
	\item[Enthalpy:] the sum of internal energy of a thermodynamic system and the work required to achieve its pressure and volume.
	\item[First law of thermodynamics:] known as Law of Conservation of Energy, that states the total energy of an isolated system is constant, and cannot be created or destroyed. 
	\item[Gibbs energy of activation:] the energy which must be provided to a chemical system with potential reactants to result in a chemical reaction.
	\item[Internal energy:] the energy contained within the thermodynamic system that can be increased by introduction of matter, by heat, or by doing thermodynamic work on the system.
	\item[Isolated system:] a thermodynamic system that cannot exchange either energy or matter with the environment.
	\item[Isothermal system:] a system with the temperature remaining constant. 
	\item[Open system:] a thermodynamic system that has external interactions taking the form of energy or material transfers into or out of the system boundary.
	\item[Port-Hamiltonian system:] Hamiltonian system that has input and out ports and satisfies Dirac structures.
	\item[Process system:] a system that transforms raw material and energy into products. 
	\item[Second law of thermodynamics:] the entropy of any isolated system always increases. 
\end{description}
\rule[1ex]{\columnwidth}{0.8pt}

\section{Related basic knowledge}\label{section appendix B}
In this appendix, we present some basic knowledge on thermodynamics and isothermal CRNT.
\subsection{Thermodynamics}{\label{section thermodynamics}}
We first review some basic knowledge on thermodynamics. Usually, there are two perspectives to study thermodynamics, the reversible pathway argument originated from Carnot \cite{carnotreflexions} and the statistical mechanics' viewpoint introduced by Gibbs\cite{gibbs1879equilibrium}, Maxwell \cite{maxwell1860ii,maxwell1860v}, and Boltzmann\cite{boltzmann2012lectures}.
In this part, we review the thermodynamics from the viewpoint of statistical mechanics, mainly referring to Sethna's textbook \cite{sethna2006statistical}. 
Notably, we assume, in this paper, the system to have a fixed volume and, for simplicity, the volume to be unitary.

\subsubsection{Thermodynamic potentials}
First, we consider a closed system where $n$ kinds of ideal fluids are well mixed.
In such a system, We denote amounts of these substances as $N_1, N_2, \dots, N_n$ (with the unit of mole) and the temperature as $T$ (with the unit of Kelvin, $K$).
According to the canonical ensemble (c.f. \cite{sethna2006statistical}), each molecule of $i$-th substance is behaving randomly with respect to the Maxwell-Boltzmann distribution 
\begin{equation}{\label{eq. Boltzmann distribution}}
\mathbb{P}(d \omega_{i})=\frac{e^{-\frac{H_{i}(\omega_i)}{RT}}}
{Z_i(T) } d \omega_{i},
\end{equation}
where $\omega_{i}$ is the microscopic state of the molecule taking value in a metric space $\Omega_{i}$, the function $H_{i}(\cdot)$ is a non-negative function called the Hamiltonian, the parameter $R$ is the Boltzmann constant equaling to $1.38 \times10^{-23} J/K$, and $Z_{i}(T)$ is the normalization factor (or called the partition function of the single molecule system) defined by 
\begin{equation}{\label{eq. the partition function}}
Z_{i}(T)\triangleq\int_{\Omega_i} \text{e}^{-\frac{H_{i}(\omega_i)}{RT}} \dd \omega_i.
\end{equation}
Here, we further assume the Hamiltonian, $H_{i}(\cdot)$, to have at least linear growth, i.e.
\begin{equation}{\label{eq. at least linear growth of the Hamiltonian}}
H_i(\omega_i) \geq \mathfrak{C}on \cdot \|\omega_i\| \text{~for all $i=1,\dots,n$ and some positive constant $\mathfrak{C}on$,}
\end{equation}
so that the integral in \eqref{eq. the partition function} is convergent.
The Maxwell-Boltzmann distribution \eqref{eq. Boltzmann distribution} works for both classical mechanics and quantum mechanics.
In classical mechanics, $\omega_{i}$ is the vector of positions and momentums of the molecule with values taken in a high dimensional Euclidean space (i.e. $\Omega_i=\mathbb{R}^d$ where $d$ is the dimension of the space).
Whereas in quantum mechanics, $\omega_{i}$ is a wave function representing the probability distribution of positions and momentums of the molecule, and the topology of the space $\Omega_i$  is much more complicated.

According to the Maxwell-Boltzmann distribution, each molecule has the average energy (with the unit of joule per mole $J/ \text{mol}$)
\begin{equation}{\label{eq. energies of individual molecules}}
u_{i}(T)\triangleq A_v \times
\frac{\int_{\Omega} H_{i}(\omega_{i}) \exp\{-\frac{H_{i}(\omega)}{RT}\} \dd \omega_{i}}{ Z_{i}(T)}, 
\end{equation}
where the integral is convergent due to \eqref{eq. at least linear growth of the Hamiltonian},  
and the whole system has the average internal energy (with the unit of $J$)
\begin{equation}{\label{eq. internal energy of the whole system}}
U(T,N)\triangleq \sum_{i=1}^{n}  N_{i} u_{i}(T) = N^{\top} u(T),
\end{equation}
where $A_v$ is the Avogadro number (a constant), $N=(N_1,N_2,...,N_n)^{\top}$, and $u(T)\triangleq ( u_{1}(T),u_{2}(T),\cdots, u_{n}(T))^{\top}$. 
The system's internal energy is the sum of the energy of each molecule in the system and, therefore, fluctuates over time due to the randomness of each molecule.
However, by the law of large numbers, the fluctuation of the internal energy is extremely small compared to the average of the internal energy, and, therefore, from the viewpoint of statistics, the internal energy can be expressed by $U(T,N)$, a deterministic function of the temperature and mass amounts. 

In statistic mechanics, the partition function is another important concept worth reviewing.
For each molecular system, the partition function is defined by \eqref{eq. the partition function}. 
For the whole system, the partition function is established based on the ones of molecular systems and has the expression 
\begin{equation*}
Z(T,N)\triangleq \prod_{i=1}^{n} \frac{\left(Z_{i}(T)\right)^{A_v N_i}}{(A_v N_i) !},
\end{equation*}
where the factorial term in the denominator is due to the indistinguishability of particles of each substance \cite{sethna2006statistical}.  
Particularly, the functions $Z_{i}(T)$ and $Z(T, N)$ are both non-descending with respect to the temperature component.
\begin{proposition}{\label{pro. partition function's monotone}}
	If the inequality \eqref{eq. at least linear growth of the Hamiltonian} holds, then Partition functions, $Z_{i}(T)$ and $Z(T,N)$, are non-decreasing with respect to the temperature component.
\end{proposition}
\begin{proof}
	Note that by definition $Z(T,N)$ is non-descending with respect to the temperature if each $Z_{i}(T)$ is non-descending. 
	Therefore, we only need to show $Z_{i}(T)$ ($i=1,\dots,n$) to be non-decreasing.
	The derivative of $Z_{i}(T)$ is given by $$\frac{\dd Z_i(T)}{\dd T}=\int_{\Omega_i}\frac{H_{i}(\omega)}{RT^2} e^{-\frac{H_{i}(\omega_i)}{RT}} \dd \omega_i\geq0,$$ where the convergence of the integral is guaranteed by \eqref{eq. at least linear growth of the Hamiltonian}. 
	Therefore, $Z_{i}(T)$ is non-descending, which proves the result.
\end{proof}

Moreover, in this paper, we adopt the following assumption, which suggests each partition function $Z_i(T)$ to grow to infinity at the infinite temperature.
\begin{condition}{\label{condition. 1}}
	Each partition function $Z_i(T)$ satisfies $\lim_{T\to \infty} Z_{j} (T) = \infty$.
\end{condition}
This assumption is not very restrictive as it holds true in many well-known cases, such as the ideal gas molecules where the Hamiltonian has a quadratic form, and, therefore, its partition function $Z_i(T)$ goes to infinity \cite{sethna2006statistical}. 
\cref{condition. 1} is mainly applied to showing the existence and uniqueness of a detailed balanced equilibrium, (see \cref{section asymptotic stability} and \cref{section proof of asymptotic stability}).

For the $i$-th species, each molecular system's Helmholtz free energy and entropy are given by 
\begin{equation}\label{gGao}
g_{i}(T)\triangleq -A_vRT\ln Z_{i}(T) \quad\text{and} \quad
s_{i}(T) \triangleq -\frac{\dd g_{i}(T)}{\dd T}
\end{equation}
with the units of $J/\text{mol}$ and $J/(\text{mole}\cdot K)$, respectively \cite{sethna2006statistical}.
For simplicity, we denote $g(T)\triangleq \left(g_1(T),g_2(T),...,g_n(T)\right)^{\top}$ and $s(T)\triangleq \left(s_1(T),s_2(T),...,s_n(T)\right)^{\top}$ in the context of this paper.
Through simple calculation, one can arrive at the following expression
\begin{equation}{\label{eq small s help}}
s_{i}(T) = \frac{u_{i}(T)-g_i(T)}{T}
\end{equation}
which will be frequently used in later analyses.
Moreover, according to L'Hospital's rule, the Helmholtz free energy satisfies the low temperature limit 
\begin{equation}{\label{eq. low temperature of free energy}}
\lim_{T\to 0} g_{i}(T)=u_{i}(0).
\end{equation}
For the whole system, the Helmholtz free energy (with the unit of $J$) is defined by
\begin{eqnarray}
G(T,N) &\triangleq& -RT\ln Z(T,N) \notag\\
&=& \sum_{i=1}^{n} \left[
\left(g_{i}(T) -A_vRT \right)N_i + A_vRT N_i \ln N_i  + \mathcal{O}\left( RT\ln ( A_v N_i)\right)  \notag
\right],
\end{eqnarray}
where $\mathcal{O}\left( RT\ln ( A_v N_i)\right)$ represents a higher order term whose order of magnitude is at most the same as the one of $RT\ln ( A_v N_i)$, and  the second equality follows immediately from Stirling's approximation \cite{sethna2006statistical}.
Since the term $RT\ln ( A_v N_i)$ is very small in the molar scale and at the real-world temperature, the higher order term in the above formula can be neglected, and, therefore, the free energy is expressed as 
\begin{equation}{\label{eq. Gibbs free energy whole system}}
G(T,N)\triangleq \sum_{i=1}^{n} \left[
\left(g_{i}(T) -A_vRT \right)N_i + A_vRT N_i \ln N_i   
\right].
\end{equation}
Moreover, the whole system's entropy (with the unit of $J\cdot\text{mole}/K$) is defined by
\begin{equation}{\label{eq. entropy function whole system}}
S(T,N)\triangleq-\frac{\partial G(T,N)}{\partial T}
= \sum_{i=1}^{N} \left[
\left(s_i(T)+A_vR \right)N_i - A_vR N_i \ln N_i  
\right],
\end{equation}
and, by \eqref{eq small s help}, \eqref{eq. Gibbs free energy whole system}, and \eqref{eq. entropy function whole system}, there holds the relation
\begin{equation}{\label{eq. small S help}}
S(T,N)=\frac{U(T,N)-G(T,N)}{T}.
\end{equation}

\subsubsection{Heat capacity and the fundamental equation}
The heat capacity is the amount of heat supplying to a system to raise the temperature by a unit degree. 
For the molecule system (of $i$-th substance) and the whole system, the heat capacity is defined respectively as
\begin{equation}{\label{eq. Definition of heat capacity}}
c_i(T)\triangleq\frac{\dd u_{i}(T)}{\dd T}
\quad
\text{and}\quad
C(T,N)\triangleq \frac{\partial U(T,N)}{\partial T}
= N^{\top} c(T),
\end{equation}
where $c(T)=(c_1(T),...,c_n(T))^\top$. 
Since the heat capacity of each molecule system can be seen as the variance of the Hamiltonian with respect to Boltzmann distribution \cite{sethna2006statistical}, heat capacities, $c_i(T)$ and $C(T,N)$, are both non-negative functions.
Moreover, by the at least linear growth of each Hamiltonian, the heat capacity $c_i(T)$ is strictly positive at any positive temperature.
\begin{proposition}{\label{pro. positivity of heat capcities}}
	If the inequality \eqref{eq. at least linear growth of the Hamiltonian} holds, then
	\begin{enumerate}
		\item  $c_j(T)>0$ for any $T>0$ and $i=1,\dots,n$,
		\item $C(T,N)>0$ for any $T>0$ and $N\in\mathbb{R}^{n}_{>0}$.
	\end{enumerate}
\end{proposition}
\begin{proof}
	The inequality \eqref{eq. at least linear growth of the Hamiltonian} suggests the function $H_i(\cdot)$ to be \textit{not} an almost everywhere constant function, and, therefore, the heat capacity $c_j(T)$, the variance of $H_i(\cdot)$ with respect to the Maxwell-Boltzmann distribution, is great than 0 at any positive temperature. 
	Therefore,  the first result is shown. 
	The second result follows immediately from the first result and the definition \eqref{eq. Definition of heat capacity}.
\end{proof}

For a thermodynamic system with $n$ components, Callen's first postulate \cite{callen1998thermodynamics} states that $n+2$ extensive variables, $(U,V,N_1,\dots,N_n)$ with $V$ the volume, define all the macroscopic properties of the system.
Recall that we fixed the system's volume in this paper. 
Therefore, Callen's first postulate suggests the system we consider can be fully characterized by $n+1$ extensive variables, $(U,N_1,\dots,N_n)$.
The relations between other thermodynamic quantities (for instance, the temperature and the entropy) and these $n+1$ extensive variables are shown in the following.

\begin{proposition}{\label{pro. fundamental equation}}
	If the inequality \eqref{eq. at least linear growth of the Hamiltonian} holds, then for any  $$(U,N)^{\top} \in \{ (U,N)^{\top} ~ | ~ U > N^\top u(0)\}\bigcap \mathbb{R}^{n+1}_{>0}$$ there hold
	\begin{equation}{\label{eq. differnetiating T}}
	\dd T= \frac{1}{C(T)} \dd U -\frac{u(T)}{C(T)} \dd N
	\end{equation} 
	and
	\begin{equation}{\label{eq. fundamental equation}}
	\dd S= \frac{1}{T} \dd U -\frac{\mu(U,N)}{T} \dd N,
	\end{equation}
	where $\mu(U,N)=g(T)+A_vRT\Ln N$.
\end{proposition}
\begin{proof}
	According to \cref{pro. positivity of heat capcities}, we have $\frac{\partial U(T,N)}{\partial T} =  C(T,N) >0$ for all $(T,N)^\top\in\mathbb{R}^{n+1}_{>0}$.
	Thus by the implicit function theorem, the relation \eqref{eq. differnetiating T} follows immediately from differentiating \eqref{eq. internal energy of the whole system} with the definition region $ \{(U,N)^{\top}~|~U> N^\top u(0)\}\bigcap \mathbb{R}^{n+1}_{>0}$.
	Moreover, by differentiating \eqref{eq. entropy function whole system}, we have the relation
	\begin{eqnarray}
	\dd S&=& 
	N^\top \frac{\dd \left[ (u(T)-g(T))/{T}\right]}{\dd T} \dd T
	+ \left[s(T)-A_vR\Ln N\right]^{\top} \dd N \notag\\
	&=& \frac{ N^{\top} c(T)}{T} \dd T + \left[s(T)-A_vR\Ln N\right]^{\top} \dd N, \notag
	\end{eqnarray}
	where the first line follows from \eqref{eq small s help}.
	By plugging \eqref{eq. differnetiating T} into the last line, we arrive at \eqref{eq. fundamental equation}.
\end{proof}

The equation \eqref{eq. fundamental equation} is the fundamental thermodynamic relation when the volume is fixed, and the term $\mu$ in \eqref{eq. fundamental equation} is called the chemical potential, which determines the direction of reactions \cite{kondepudi2014modern}. 
Moreover, from the relation \eqref{eq. fundamental equation}, we can observe that the entropy function $S(U,N)$ is strictly concave.

\begin{proposition}{\label{pro. strict concavity}}
	If the inequality \eqref{eq. at least linear growth of the Hamiltonian} holds, then $S(U,N)$ is strictly concave in the region $\{(U,N)^{\top}~|~U> N^\top u(0)\}\bigcap \mathbb{R}^{n+1}_{>0}$.
\end{proposition}
\begin{proof}
	By \eqref{eq. fundamental equation} and \eqref{eq. differnetiating T}, we can calculate the Hessian matrix of $S(U,N)$ as 
	\begin{equation}
	\mathcal{H}(U,N)
	\triangleq
	\left(
	\begin{array}{cc}
	\frac{\partial^2 S}{\partial U^2} & \frac{\partial^2 S}{\partial U \partial N} \\
	\left(\frac{\partial^2 S}{\partial U \partial N}\right)^{\top} &
	\frac{\partial^2 S}{\partial N^2}
	\end{array}
	\right)\notag 
	=	 
	\frac{-1}{T^2C(T)}
	\left(
	\begin{array}{cc}
	1 & -u^{\top}(T) \notag\\
	-u(T) & \mathcal{D}(N)+ u(T) u^{\top}(T)
	\end{array}
	\right),
	\end{equation}
	where $\mathcal{D}(N)=A_v R \cdot \text{diag}\{\ln N_1, \ln N_2, \dots, \ln N_n\}$.
	Obviously, $\mathcal{D}(N)$ is a positive diagonal matrix.
	Moreover, by denoting the matrix
	\begin{equation*}
	\mathcal{P}\triangleq
	\left(
	\begin{array}{cc}
	1 &  -u(T) \\
	& \mathcal{I}_{n}
	\end{array}
	\right),
	\end{equation*}
	we arrive at the relation
	\begin{equation}
	\mathcal{H}(U,N) = 
	-\frac{1}{T^2C(T)}
	\mathcal{P}^{T} 
	\left(
	\begin{array}{cc}
	1 &  \notag\\
	& \mathcal{D}(N)
	\end{array}
	\right)
	\mathcal{P}
	<0,
	\end{equation}
	which proves this proposition. 
\end{proof}

In the context of this paper, we call a state $(U,N)$ positive if $(U,N)^{\top}\in \mathbb{R}^{n}_{>0}$ and $U>N^{\top}u(0)$, i.e., the internal energy, mass amounts, and temperature are all positive.

\subsection{Isothermal CRNT}\label{iso_CRNT}
We then review basic concepts and terminologies about isothermal CRN. 

Consider an isothermal CRN system with $n$ substances which are well mixed and a unit volume. Based on Callen's first postulate \cite{callen1998thermodynamics}, the system can be fully described by the amounts of $n$ substances, i.e., $N=(N_1,...,N_n)$. Further, assume that there are $r$ reactions taking place among $n$ substances with the reaction scheme
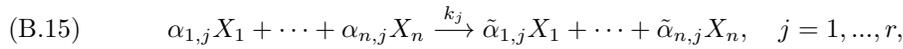
\begin{equation}{\label{eq. isothermal chemical reaction scheme}}
\alpha_{1,j} X_{1}+ 
\cdots + \alpha_{n,j} X_{n} \stackrel{k_j}{\longrightarrow}    
\tilde\alpha_{1,j} X_{1}+ 
\cdots+ \tilde\alpha_{n,j} X_{n}, \quad j=1,...,r,
\end{equation}
where $X_i$ represents the $i$-th species in the system, $\alpha_{i,j}$ and $\tilde\alpha_{i,j}$ are non-negative integers called stoichiometric coefficients, the integer vectors, $(\alpha_{1,j},\dots,\alpha_{n,j})^{\top}$ and $(\tilde \alpha_{1,j},\dots,\tilde \alpha_{n,j})^{\top}$, are termed as complexes, and $k_j$ are reaction rate constants.
The isothermal CRNT \cite{feinberg1987chemical,feinberg1995existence,horn1972general} views the above reaction scheme as a graph, where the linear combinations of species on both sides of the reactions are the vertexes, the reaction arrows are the edges, and reaction rate constants are weights associated with edges. Mathematically, the isothermal CRN is proposed by a quadruplet $(\mathcal{S},\mathcal{C},\mathcal{R},\mathcal K)$ with the definition as follows. 
\begin{definition}[Isothermal CRN \cite{feinberg1987chemical,feinberg1995existence,horn1972general}]{\label{def. isothermal crn}}
	An isothermal CRN is quadruplet $(\mathcal{S},\mathcal{C},\mathcal{R},\mathcal K)$, in which 
	\begin{itemize}
		\item  $\mathcal{S}\triangleq\{X_1,\dots, X_n\}$ is the species set representing the considered substances,
		\item $\mathcal{C}\triangleq\{y_1,\dots, y_m\}$ is the complex set consisting of all distinguished complexes where $m$ is the size of the set.
		\item $\mathcal{R}\triangleq\{y_{\sigma_{1}}\to y_{\pi_{1}},\dots,y_{\sigma_{r}}\to y_{\pi_{r}}\}$ is the reaction set with $\sigma_{j}$ and $\pi_{j}$ $(1\leq j\leq r)$ being respectively the indexes of the substrate and product complexes of the $j$-th reaction. Particularly, for each reaction, $\sigma_{j}$ and $\pi_{j}$ should not be the same.
		\item $\mathcal{K}\triangleq(k_{1},\dots,k_{r})$ is the kinetics class.
	\end{itemize}
\end{definition}

Utilizing the graph theory, recent literature \cite{van2015complex,rao2013graph,van2016network} introduce the \textit{complex matrix}, $Y\triangleq(y_1,\dots,y_{m})\in\mathbb{R}^{n\times m}$ whose $j$th column corresponds to the $j$th complex in $\mathcal C$, and the \textit{incidence matrix}, $D\in\mathbb{R}^{m\times r}$ where
\begin{equation}{\label{eq. incidence matrix}}
D_{i,j}\triangleq\left\{
\begin{array}{cc}
1,   & i=\pi_{j},  \\
-1,  & i=\sigma_{j}, \\
0,   & \text{otherwise},
\end{array}
\right.
\quad i\in\{1,\dots,m\} \text{~and~} j\in\{1,\dots,r\}.
\end{equation}
The former corresponds to the complexes (vertexes) while the latter shows linkages between these complexes (i.e., edges). They jointly define the stoichiometric matrix $\Gamma$ by $\Gamma\triangleq YD$. When equipped with mass-action kinetics, the reaction rate, denoted by $v_j(\cdot)$ for the $j$th reaction, may be evaluated according to the power law with respect to the concentrations of substances, i.e., 
\begin{equation}{\label{eq. mass action kinetics}}
v_{j}(N) = k_{j} {N}^{y_{\sigma_{j}}}\footnote{
	In the usual expression of mass-action kinetics, the reaction rate is proportional to the product of concentrations. 
	As we restrict the system volume to be fixed and unitary, the concentrations and amounts of substances are numerically the same, and, thus, \eqref{eq. mass action kinetics} does not violate the usual expression.}
= k_{j}\Exp{\left(y_{\sigma_{j}}^{\top} \ln N\right)}, \qquad j=1,\dots,r,
\end{equation}
Therefore, the dynamics of the system has the form
\begin{equation}{\label{eq. dynamics of isothermal CRN}}
\dot N 
=YDv(N)=\Gamma v(N),
\end{equation}
where $v(N)=\left(v_{1}(N),\dots,v_{r}(N)\right)$.
From \eqref{eq. dynamics of isothermal CRN}, one can observe that the increment of the state, $N(t)-N(0)$, where $t$ and $0$ are the time argument, always belongs to the linear space $\text{Im~}  \Gamma$, and, therefore, the state can only evolve in an invariant set $\left(N(0)+\text{Im~} \Gamma \right) \bigcap \mathbb{R}^{n}_{\geq 0}$.
The isothermal CRNT terms the linear space $\text{Im~}  \Gamma$ as the {\em stoichiometric subspace}, the invariant sets $\left(N(0)+\text{Im~} \Gamma \right) \bigcap \mathbb{R}^{n}_{\geq 0}$ as the {\em stoichiometric compatibility class}, 
and the positive part of the invariant set, $\left(N(0)+\text{Im~} \Gamma \right) \bigcap \mathbb{R}^{n}_{> 0}$, as the {\em positive stoichiometric compatibility class}.

Notably, along with real chemical reactions, CRN can also model boundary fluxes by using zero complexes to represent environment (see \cite[section 5]{van2013mathematical}), which makes the theory very useful in both chemical engineering and biology applications.

A mass-action CRN $(\mathcal{S},\mathcal{C},\mathcal{R},\mathcal{K})$ is {\em reversible} if  forward and backward reactions $y_{\sigma_{j}}\to y_{\pi_{j}}$ and $y_{\pi_{j}}\to y_{\sigma_{j}}$ come in pairs.
Among all chemical reaction networks, detailed balanced networks are special ones, where the network is reversible, and at some equilibrium, reaction rates of each pair of reversible reactions are balanced.
In isothermal CRNT, such equilibria that balance the forward and backward reactions are termed as detailed balanced equilibria. 
In the past few decades, a large body of literature has investigated detailed balanced networks; a summary of the fundamental properties of them is listed as follows.

\begin{theorem}[\cite{feinberg1995existence,van2013mathematical}]{\label{thm. isothermal detailed balanced}}
	Let a mass-action isothermal CRN $(\mathcal{S},\mathcal{C},\mathcal{R},\mathcal{K})$ be detailed balanced and possess a positive detailed balanced equilibrium $N^*$. By denoting $${G}_{\mathcal A}(N) \triangleq N^{\top} \left( \Ln{{N}} - \Ln N^* \right)-\mathbbold{1}_{n}^{\top}\left(N-N^*\right),$$ the following properties hold.
	\begin{enumerate}
		\item ${G}_{\mathcal A}(N)$ (as a function of $N$) serves as a Lyapunov function for the system; $\dot {G}_{\mathcal A}(N)\leq 0$ where the equality holds if and only if $\nabla G (N) \in \ker (\Gamma^{\perp})$
		\item Any positive equilibrium of the system is detailed balanced. 
		\item In each positive stoichiometric compatibility class, there exists only one equilibrium, and this equilibrium is locally asymptotically stable. 
	\end{enumerate}
\end{theorem}

The Lyapunov function ${G}_{\mathcal A}(N)$ first introduced by Horn and Jackson \cite{horn1972general} is called the pseudo-Helmholtz free energy because of the similarity between these two functions' expressions.
Specifically,we can easily observe by \eqref{eq. Gibbs free energy whole system} that 
\begin{equation}{\label{eq. Availability function explanation for pseudo-Helmholtz free energy}}
A_vRT^*G_{\mathcal A}(N)=G(T^*,N)-\frac{\partial G(T^*,N^*)}{\partial N}\left(N-N^*\right)-G(T^*,N^*), 
\end{equation}
where $T^*$ is the temperature of the isothermal system, $G(T^*,N)$ is the Helmholtz free energy (see \eqref{eq. Gibbs free energy whole system}), and the term on the right-hand side of the equality is the availability function of the Helmholtz free energy with respect to mass amounts (c.f. \cite{alonso2001stabilization}).

\section{Proofs of the main results}\label{section proof}
We give the detailed proofs of our two main theorems, \cref{thm. stability} and \cref{thm. asymptotic stability}, in this appendix. 

\subsection{The proof of \cref{thm. stability}}{\label{section proof of stability}}
We first introduce basic concepts and properties of Laplacian matrix.

\begin{definition}{\label{def. Laplacian matrix}}
	A matrix $L\in\mathcal{R}^{m\times m}$ is a balanced Laplacian matrix if i) all its diagonal elements are positive, ii) all its off-diagonal elements are non-positive, and iii) its column sums and row sums are zero, i,e, $\mathbbold{1}^{\top}_{m}L=\mathbbold{0}^{\top}_{m}$ and $L \mathbbold{1}_{m}=\mathbbold{0}_{m}$.
\end{definition}

\begin{proposition}{\label{pro. Laplacian matrix}}
	The matrices, $B_{\rm CR} K_{\rm CR}(T)B_{\rm CR}^{\top}$ and $B_{\rm IO} K_{\rm IO}B_{\rm IO}^{\top}$, are both balanced Laplacian matrices when $T>0$. 
\end{proposition}
\begin{proof}
	For $B_{\rm CR} K_{\rm CR}(T)B_{\rm CR}^{\top}$ where $T>0$, each diagonal element satisfies
	\begin{equation*}
	\left(B_{\rm CR} K_{\rm CR}(T)B_{\rm CR}^{\top}\right)_{i,i}
	= B_{i \cdot } K_{\rm CR}(T) B_{i \cdot}^{\top} >0,  \quad i=1,\dots, m,
	\end{equation*}
	where the inequality follows from the positive definiteness of $K_{\rm CR}(T)$, and each off-diagonal element $\left(B_{\rm CR} K_{\rm CR}(T)B_{\rm CR}^{\top}\right)_{i_1,i_2}$ ($i_1,i_2=1,\dots, m$ and $i_1\neq i_2$) satisfies
	\begin{align*}
	\left(B_{\rm CR} K_{\rm CR}(T)B_{\rm CR}^{\top}\right)_{i_1,i_2}
	= B_{i_1 \cdot } K_{\rm CR}(T) B_{i_2 \cdot}^{\top} =\sum_{j=1}^{r_{\rm CR}/2} B_{i_1 j }B_{i_2 j} \left(K_{\rm CR}(T)\right)_{j,j}
	\leq 0,
	\end{align*}
	where the inequality follows from $B_{i_1 j }B_{i_2 j}\leq 0$ for $i_1\neq i_2$ (c.f. \eqref{eq. definition of matrix B}).
	Therefore, the first two conditions in \cref{def. Laplacian matrix} hold.
	Moreover, by the definition of $B$, it follows the relation $\mathbbold{1}^{\top}_{m} B = \mathbbold{0}^{\top}_{m}$, and, therefore, $\mathbbold{1}^{\top}_{m}B_{\rm CR} K_{\rm CR}(T)B_{\rm CR}^{\top}=\mathbbold{0}^{\top}_{m}$ 
	and $B_{\rm CR} K_{\rm CR}(T)B_{\rm CR}^{\top}\mathbbold{1}_{m}=\mathbbold{0}_{m}$, which suggests the third condition in  \cref{def. Laplacian matrix} to hold.
	Given the analysis above, we can conclude that $B_{\rm CR} K_{\rm CR}(T)B_{\rm CR}^{\top}$ is a Laplacian matrix when $T>0$. 
	
	Similarly, we can prove the result for $B_{\rm IO} K_{\rm IO}B_{\rm IO}^{\top}$, which shows the proposition.
\end{proof}

\begin{proposition}{\label{pro. kernel space}}
	If $T>0$, then there hold $\ker \left(B_{\rm CR} K_{\rm CR}(T)B_{\rm CR}^{\top}\right)= \ker B_{\rm CR}^{\top} $ and $\ker \left(B_{\rm IO} K_{\rm IO}B_{\rm IO}^{\top}\right)= \ker B_{\rm IO}^{\top} $.
\end{proposition}
\begin{proof}
	If $\gamma\in \ker (B_{\rm CR} K_{\rm CR}(T)B_{\rm CR}^{\top})$, then $\gamma^{\top}\ker B_{\rm CR} K_{\rm CR}(T)B_{\rm CR}^{\top} \gamma =0$.
	By the positive definiteness of $K_{\rm CR}(T)$ when $T>0$, we have the relation $B_{\rm CR}^{\top} \gamma =\mathbbold{0}_{r_{\rm CR}}$, which suggests $\ker (B_{\rm CR} K_{\rm CR}(T)B_{\rm CR}^{\top}) \subset \ker (B_{\rm CR}^{\top})$.
	Since one can easily observe the relation $\ker (B_{\rm CR}^{\top}) \subset \ker (B_{\rm CR} K_{\rm CR}(T)B_{\rm CR}^{\top})$, we therefore can conclude $\ker (B_{\rm CR}^{\top}) = \ker (B_{\rm CR} K_{\rm CR}(T)B_{\rm CR}^{\top})$.
	
	Similarly, we can prove $\ker (B_{\rm IO} K_{\rm IO}B_{\rm IO}^{\top})= \ker (B_{\rm IO}^{\top} )$, which shows the result.
\end{proof}

\begin{lemma}[{\cite{van2015complex,rao2013graph}}]\label{lemma Laplacian matrix}
	If $L$ is a balanced Laplacian matrix, then $\gamma ^{\top}L\Exp(\gamma)\geq 0$ for any $\gamma\in\mathbb{R}^n$, where the equality holds if $L^{\top}\gamma=0$.
\end{lemma}

\begin{proposition}{\label{pro. a negative vector}}
	If \eqref{eq. at least linear growth of the Hamiltonian} and $r_{\rm IO}\neq 0$ hold, then each component of the vector $ \left(\frac{1}{T^*}-\frac{1}{T}\right)\Delta \mathcal{U}_{\rm IO} +\left(\frac{g(T)}{T}-\frac{g(T^*)}{T^*}\right)^{\top}YD_{\rm IO} $ is zero when $T= T^*$ and otherwise negative.
\end{proposition}
\begin{proof}
	By denoting the $j$-th element of the considered vector as $\xi_{j}(T)$ , we have
	\begin{equation*}
	\xi_{j}(T)={-u_{i}(T)}/{T^*}+{u_{i}(T)}/{T} -{g_i(T)}/{T}+{g_i(T^*)}/{T^*}, 
	\quad \text{for}~ y_{\sigma_{j+r_{\rm CR}}}=\delta_{i},
	\end{equation*}
	and, by \eqref{eq. temperature at detailed balanced equilirbium}, 
	\begin{equation*}
	\xi_{j}(T)={u_{i}(T^*)}/{T^*}-{u_{i}(T^*)}/{T} +{g_i(T)}/{T}-{g_i(T^*)}/{T^*},
	\quad \text{for}~ y_{\pi_{j+r_{\rm CR}}}=\delta_{i}.
	\end{equation*}
	According to above equations, we can observe that all $\xi_{j}(T)$ are zero when $T=T^*$ . Moreover, by \eqref{eq small s help}, we can easily calculate the derivative of $\xi_{j}(T)$ as
	\begin{equation*}
	\frac{\dd \xi_{j} (T)}{\dd t} =
	\left\{
	\begin{array}{lr}
	\frac{c_{i}(T)}{T}-\frac{c_{i}(T)}{T*}, & y_{\sigma_{j+r_{\rm CR}}}=\delta_{i}, \\
	\frac{u_i(T^*)-u_i(T)}{T^2}, & y_{\pi_{j+r_{\rm CR}}}=\delta_{i},
	\end{array}
	\right.
	\end{equation*}
	which suggests the derivative of $\xi_{j}$ to be zero at $T=T^*$, positive for $T<T^*$, and negative for $T>T^*$.
	Given the analysis above, we can conclude that each $\xi_{j}(T)$ is zero at $T=T^*$ and otherwise negative, which shows the result. 
\end{proof}

With these preparation, we can finally prove \cref{thm. stability}.

\begin{proof}[The proof of \cref{thm. stability}]
	According to the chain rule and the dynamics \eqref{eq. compact formula}, it follows that
	\begin{align}
	&\dot{S}_{\mathcal A}(U, N) \notag \\
	&=- \left( \frac{\mu}{ T} - \frac{\mu^*}{T^* } \right)^{\top}YB_{\rm CR} K_{\rm CR}(T)B_{\rm CR}^{\top} \Exp\left(Y \left( \frac{\mu}{A_v R T} - \frac{\mu^*}{A_v RT^*} \right) \right) \notag\\
	&\quad -A_v R \left( \Ln N-\Ln N^* \right)^\top YB_{\rm IO} K_{\rm IO}B_{\rm IO}^{\top} \Exp\left(Y\left( \Ln N-\Ln N^* \right)  \right) \notag\\
	&\quad +\left[ \left({1}/{T^*}-{1}/{T}\right)\Delta \mathcal{U}_{\rm IO} +\left({g(T)}/{T}-{g(T^*)}/{T^*}\right)^{\top}YD_{\rm IO} \right] v_{\rm IO} \notag \\
	&\quad  + ({1}/{T^*}-{1}/{T})\Delta \mathcal{U}_{\rm CR} v_{\rm CR}(U,N)-\tilde k_{r_{\rm CR}+r_{\rm IO}+1} h\left(T^*-T\right)^2 \chi(r_{\rm HE} \neq 0)/(T^*T)\notag \\
	&\leq \left[ \left({1}/{T^*}-{1}/{T}\right)\Delta \mathcal{U}_{\rm IO} +\left({g(T)}/{T}-{g(T^*)}/{T^*}\right)^{\top}YD_{\rm IO} \right] v_{\rm IO} \label{eq. SA dot} \\
	&\quad  + ({1}/{T^*}-{1}/{T})\Delta \mathcal{U}_{\rm CR} v_{\rm CR}(U,N)  -\tilde k_{r_{\rm CR}+r_{\rm IO}+1} h\left(T^*-T\right)^2 \chi(r_{\rm HE} \neq 0)/(T^*T),\notag 
	\end{align}
	where the inequality follows from \cref{pro. Laplacian matrix} and \cref{lemma Laplacian matrix}, and the equality holds if and only if $(\mu/T-\mu^*/T^*) \in \ker (B^{\top}_{\rm CR}Y^{\top})$ and $(\Ln N-\Ln N^*) \in \ker (B^{\top}_{\rm IO}Y^{\top})$ (by \cref{lemma Laplacian matrix} and \cref{pro. kernel space}).
	According to \cref{condition. 3}, there are several cases, we need to analyze separately:
	\begin{enumerate}
		\item For $\Delta U_{\rm CR}\neq \mathbbold{0}_{r_{\rm CR}}$ and $r_{\rm CR}=r_{\rm IO}=0$ (i.e. the second case in \cref{condition. 3}), the inequality \eqref{eq. SA dot} suggests that $\dot{S}_{\mathcal A}(U, N) \leq 0$ at any positive state with $T=T^*$ where the inequality holds when $(\mu/T-\mu^*/T^*) \in \ker (B^\top_{\rm CR}Y^{\top})$.  Since $\ker  (B_{\rm CR}^{\top}Y^{\top})=\ker (B^{\top}Y^{\top})=\ker  (\Gamma^{\top})$ (by \cref{prop. stoichiometric compatibility class}) hold in this case, we can further conclude $\dot{S}_{\mathcal A}(U, N) \leq 0$ at any positive state with $T=T^*$ where the inequality holds if and only if $\nabla{S}_{\mathcal A}(U, N)
		\in \ker \tilde{ \Gamma}^{\top}$.
		\item For $\Delta U_{\rm CR}= \mathbbold{0}_{r_{\rm CR}}$ and $r_{\rm CR}=r_{\rm IO}=0$ (i.e. a part of the first case in \cref{condition. 3}), the inequality \eqref{eq. SA dot} suggests that $\dot{S}_{\mathcal A}(U, N) \leq 0$ at any positive state where the equality holds if and only if $(\mu/T-\mu^*/T^*) \in \ker (B_{\rm CR}^{\top} Y^{\top})$.  Since $\ker  (B_{\rm CR}^{\top}Y^{\top})=\ker (B^{\top}Y^{\top})=\ker  (\Gamma^{\top})$  (by \cref{prop. stoichiometric compatibility class}) hold in this case, we can further conclude $\dot{S}_{\mathcal A}(U, N) \leq 0$ at any positive state where the inequality holds if and only if $\nabla{S}_{\mathcal A}(U, N)
		\in \ker (\tilde{ \Gamma}^{\top})$.
		\item For $\Delta U_{\rm CR}= \mathbbold{0}_{r_{\rm CR}}$ and at least one of $r_{\rm CR}$ and $r_{\rm IO}$ being non-zero (i.e. the other part of the first case in \cref{condition. 3}), the inequality \eqref{eq. SA dot} suggests by \cref{pro. a negative vector} that $\dot{S}_{\mathcal A}(U, N) \leq 0$ at any positive state where the equality hold if and only if $(\mu/T-\mu^*/T^*) \in \ker (B_{\rm CR}^{\top}Y^{\top})$, $(\Ln N-\Ln N^*) \in \ker (B_{\rm IO}^{\top} Y^{\top})$ and $T=T^*$. 
		Note that the condition $(\Ln N-\Ln N^*) \in \ker (YB_{\rm IO})$ and $T=T^*$ suggests $(\mu/T-\mu^*/T^*)$ ($= g(T)/T-g(T^*)/T^*+A_{v}R\left(\Ln N-\Ln N^*\right)$) to be in the space $\ker (B_{\rm IO}^{\top} Y^{\top})$. 
		Therefor, we can further conclude that $\dot{S}_{\mathcal A}(U, N) \leq 0$ where the the equality holds if and only if $\frac{\partial {S}_{\mathcal A}(U, N)}{ \partial N} \in \ker (B_{\rm CR}^{\top}Y^{\top})+\ker (B_{\rm IO}^{\top}Y^{\top})=\ker (B^{\top}Y^{\top})=\ker ({ \Gamma}^{\top})$ and $\frac{\partial {S}_{\mathcal A}(U, N)}{ \partial U}=0$ (i.e. $\nabla{S}_{\mathcal A}(U, N)
		\in \ker (\tilde{ \Gamma}^{\top})$).
	\end{enumerate}
	Given the analysis above, we show the function ${S}_{\mathcal A}(U, N)$ to be dissipative.
	
	Note that, besides dissipativeness, ${S}_{\mathcal A}(U, N)$ is also positive definite in the positive stoichiometric-like compatibility class $\mathcal{PS}(U^*,N^*)$. 
	Therefore, by Lyapunov's second method, the function ${S}_{\mathcal A}(U, N)$  is a Lyapunov function for the dynamics \eqref{eq. compact formula}, rendering the state $(U^*,N^*)$ to be stable. 
\end{proof}

\subsection{The proofs of \cref{thm. asymptotic stability}} {\label{section proof of asymptotic stability}}
In this subsetion, we provide the proofs for results related to \cref{thm. asymptotic stability}.

\textit{Proof of} \cref{pro. unboundedness of LA}. We first show the unboundedness of $\mathcal{L}_{A}(\beta,\gamma)$ at any finite boundary, i.e. \eqref{eq. unboundedness of LA at finite boundary}.
According to the definition, we can re-express $\mathcal{L}_{A}(\beta,\gamma)$ as 
\begin{align}
\mathcal{L}_{A}(\beta,\gamma) & =
S(U,N)+\left(\beta -\frac{1}{T^*}\right)U+\left(\gamma+\frac{\mu^*}{T^*}\right)^{\top}N-\beta U^o -\gamma^{\top} N^o + \mathfrak{C}on \notag \\
& = S(U,N)-\frac{U}{T}+\frac{G}{T}+A_v R\sum_{i=1}^n N_i -\beta U^o -\gamma^{\top} N^o + \mathfrak{C}on \notag \\
& = A_v R\left(\sum_{i=1}^n N_i\right) -\beta U^o -\gamma^{\top} N^o + \mathfrak{C}on, \label{eq. LA reexpress}
\end{align}
where $\mathfrak{C}on$ is some constant, and the last equality follows from \eqref{eq. small S help}. 
By \eqref{eq. express of temperature by beta gamma}, we can conclude that $\lim_{\beta\to \left(\frac{1}{T^*}\right)^{-}}T= +\infty$, and, therefore, 
\begin{equation*}
\lim_{(\beta,\gamma)\to ({\beta^{b}}^{-},\gamma^{b})}
N_{i} = \lim_{T\to \infty} Z_i(T)\exp\left(\frac{\gamma^{b}+\frac{\mu^*}{T^*}}{A_v R} \right)
= +\infty, \quad \forall\left(\beta^{b},\gamma^{b}\right) \in \left\{\frac{1}{T^*}\right\} \otimes \mathbb{R}^{n},
\end{equation*}
where the first equality follows from \eqref{eq. express of mass amount by beta gamma}, and the second from \cref{condition. 1}.
By plugging it into \eqref{eq. LA reexpress}, we arrive at \eqref{eq. unboundedness of LA at finite boundary}, i.e. the unboundedness of $\mathcal{L}_{A}(\beta,\gamma)$ at any finite boundary. 

Then, we show part of \eqref{eq. unboundedness of LA at infinite boundary} where $\bar \beta=0$ and $\bar \gamma \neq \mathbbold{0}_{n}$.
In this case, we obtain 
\begin{align*}
\lim_{\theta \to +\infty} \mathcal{L}_{\mathcal{A}}(\theta \bar\beta,\theta \bar \gamma) 
&= \mathcal{L}_{\mathcal{A}}(0,\mathbbold{0}_{n})
+	\lim_{\theta \to +\infty} \int_{0}^{\theta}  \bar{\gamma}^{\top}(N-N^o) \dd \eta  \\
& = \mathcal{L}_{\mathcal{A}}(0,\mathbbold{0}_{n})
+	\lim_{\theta \to +\infty} \int_{0}^{\theta}  \bar{\gamma}^{\top}\left(\Exp\left\{\frac{\eta \bar \gamma  +\frac{\mu^*}{T^*}-\frac{g_i(T*)}{T*}}{A_v R} \right\}-N^o\right)\dd \eta  \\
&= \mathcal{L}_{\mathcal{A}}(0,\mathbbold{0}_{n}) + \lim_{\theta \to +\infty} \Bigg(
A_v R\bar{\gamma}^{\top}\Exp\left\{\frac{\theta \bar \gamma  +\frac{\mu^*}{T^*}-\frac{g_i(T*)}{T*}}{A_v R} \right\}
-\theta \bar \gamma ^{\top} N^o  \\
&\qquad \qquad \qquad \qquad\qquad\qquad \qquad \quad
-
A_v R\bar{\gamma}^{\top}\Exp\left\{\frac{\frac{\mu^*}{T^*}-\frac{g_i(T)}{T}}{A_v R} \right\}
\Bigg)\\
&=+\infty,
\end{align*}
where the second equality follows from \eqref{eq. express of temperature by beta gamma} and \eqref{eq. express of mass amount by beta gamma}, and the last equality follows from the unboundedness of function $f(x)=b_1e^{ax}-abx$, ($b_1,b_2>0$, $a\neq 0$), at infinities.

Finally, we show the rest part of  \eqref{eq. unboundedness of LA at infinite boundary} where $\bar \beta<0$.
In this case, we can calculate
\begin{align*}
\frac{\mathcal{L}_{\mathcal{A}}(\theta \bar\beta,\theta \bar \gamma) }{\dd \theta}
&= \bar \beta  (U-U^o) + \bar{\gamma}^{\top}(N-N^o) \\
& = -\bar \beta \left[U^o-(N^o)^{\top}u(0)\right] +
N_{i} 
\left(u_i(T)\bar \beta_i +\bar\gamma_i\right)
-N^o_i \left(u_{i}(0) \bar \beta_i +\bar\gamma_i \right). 
\end{align*}
We further denote ${\rm term}_i (\theta \bar\beta,\theta \bar \gamma)\triangleq	N_{i} 
\left(u_i(T)\bar \beta_i +\bar\gamma_i\right)
-N^o_i \left(u_{i}(0) \bar \beta_i +\bar\gamma_i \right) $ and,
by  \eqref{eq. express of temperature by beta gamma} and \eqref{eq. express of mass amount by beta gamma}, arrive at 
\begin{align*}
&\lim_{\theta \to \infty} {\rm term}_i (\theta \bar\beta,\theta \bar \gamma) \\
&= \lim_{\theta \to \infty} \exp\left(\frac{\theta \bar \gamma_i+\frac{\mu^*}{T^*}-\frac{g_i(T)}{T}}{A_v R} \right) 	\left(u_i(T)\bar \beta_i +\bar\gamma_i\right)
-N^o_i \left(u_{i}(0) \bar \beta_i +\bar\gamma_i \right) \\
&= \lim_{\theta \to \infty} \exp\left(\frac{\theta   \left( \bar \gamma_i+u_i(0) \bar \beta \right)+\frac{\mu^*}{T^*}}{A_v R} \right) 	\left(u_i(0)\bar \beta_i +\bar\gamma_i\right)
-N^o_i \left(u_{i}(0) \bar \beta_i +\bar\gamma_i \right),
\end{align*}
where the second equality follows from $\lim_{\theta \to \infty} T  = 0$, $\lim_{\theta \to \infty} 1/T = \lim_{\theta \to \infty} \bar \beta \theta $, and  \eqref{eq. low temperature of free energy}.
To calculate the limit of ${\rm term}_i (\theta \bar\beta,\theta \bar \gamma)$, we need to consider three cases:
\begin{enumerate}
	\item ``$u_{i}(0) \bar \beta_i +\bar\gamma_i >0$": In this case, the exponential term in the above dominants, and therefore, $\lim_{\theta \to \infty} {\rm term}_i (\theta \bar\beta,\theta \bar \gamma)=\infty$.
	\item ``$u_{i}(0) \bar \beta_i +\bar\gamma_i =0$": In this case, ${\rm term}_i (\theta \bar\beta,\theta \bar \gamma)\equiv 0$.
	\item ``$u_{i}(0) \bar \beta_i +\bar\gamma_i <0$": In this case, the constant term in the above dominants, and therefore, $\lim_{\theta \to \infty} {\rm term}_i (\theta \bar\beta,\theta \bar \gamma)>0$.
\end{enumerate}
To conclude, there exists a positive value $\bar \theta$ such that for any $\theta > \bar \theta $ each ${\rm term}_i (\theta \bar\beta,\theta \bar \gamma) > 0$, and, therefore
\begin{equation*}
\frac{\mathcal{L}_{\mathcal{A}}(\theta \bar\beta,\theta \bar \gamma) }{\dd \theta} > -\bar \beta \left[U^o-(N^o)^{\top}u(0)\right], \quad  \forall \theta > \bar \theta.
\end{equation*}
This fact suggests $\mathcal{L}_{\mathcal{A}}(\theta \bar\beta,\theta \bar \gamma)$ to have at least linear growth with respect to $\theta$ and go to infinity as $\theta \to +\infty$.

Given the analysis above, we prove the proposition. $\Box$

\textit{Proof of} \cref{pro. convexity and closeness of SC}.
The non-emptiness of $\mathfrak{S}_{c}$ follows from $(0,\mathbbold{0}_{n}) \in \mathfrak{S}_{c}$; the convexity follows from the convexity of the function $\mathcal{L}_{A}(\beta,\gamma)$.
Therefore, we only need to show the closeness of this set. 
Let $\{(\beta^{\ell}, \gamma ^{\ell})\}_{\ell}$ be a convergent sequence in $\mathfrak{S}_{c}$ with the limit $(\beta^{\infty}, \gamma ^{\infty})$ (in the closure of $\mathfrak{D} \cap  \ker (\tilde{\Gamma}^{\top})$).
By \eqref{eq. unboundedness of LA at finite boundary}, we know that $(\beta^{\infty}, \gamma ^{\infty})$ cannot be on the boundary of $\mathfrak{D}$ and, therefore, must in the interior of $\mathfrak{D} \cap  \ker (\tilde{\Gamma}^{\top})$.
Moreover, by the continuity, we have 
$
\mathcal{L}_{A}(\beta^{\infty},\gamma^{\infty})
=\lim_{\ell \to \infty} \mathcal{L}_{A}(\beta^{\ell},\gamma^{\ell})
\leq \mathcal{L}_{A}(0,\mathbbold{0}_{n}),
$
which suggests $(\beta^{\infty},\gamma^{\infty})$ to be in the set  $\mathfrak{S}_{c}$.
Therefore, we show the result. $\Box$

\textit{Proof of} \cref{pro. compactness of SC}.
We first prove by contradiction that $\mathfrak{S}_{c}$ contains no half line with endpoint $\mathbbold{0}_{n+1}$.	
Assume $\mathfrak{S}_{c}$ contains a half line $\{(\theta\bar\beta,\theta \bar \gamma)~|~\theta >0\}$ where $(\bar\beta,\bar \gamma)^{\top} \neq \mathbbold{0}_{n+1}$.
By \eqref{eq. unboundedness of LA at infinite boundary} and the definition of  $\mathfrak{S}_{c}$, the parameter $\bar\beta$ should be greater than zero; by the fact $ \mathfrak{D}=(-\infty,\frac{1}{T^*})\otimes \mathbb{R}^{n}$ is only half of $\mathbb{R}^{n+1}$, the parameter $\bar\beta$ cannot be greater than zero.
As a result, $\mathfrak{S}_{c}$ contains no half line with endpoint $\mathbbold{0}_{n+1}$.

Note that an unbounded closed and convex set containing $\mathbbold{0}_{n+1}$ must have a half line with endpoint $\mathbbold{0}_{n+1}$ \cite[P. 105]{stoer1970convexity}.
Therefore, the closed convex set $\mathfrak{S}_{c}$ (c.f. \cref{pro. convexity and closeness of SC}) is bounded, and, moreover, compact. $\Box$

\textit{Proof of} \cref{pro. minimum point of LA}.
The compactness of $\mathfrak{S}_{c}$ (c.f. \cref{pro. compactness of SC}) suggests the continuous function $\mathcal{L}_{A}(\beta,\gamma)$ to have a locally minimum point in the region $\mathfrak{S}_{c}$. 
Moreover, by the definition of $\mathfrak{S}_{c}$, this locally minimum value is less than any value $\mathcal{L}_{A}(\beta,\gamma)$ takes outside $\mathfrak{S}_{c}$.
Therefore, this minimum point is also a minimum in $\mathfrak{D} \cap  \ker (\tilde{\Gamma}^{\top})$.

The uniqueness of the minimum point follows immediately from the strict convexity of $\mathcal{L}_{A}(\beta,\gamma)$, and, therefore, we prove the result. $\Box$

\textit{Proof of} \cref{thm. asymptotic stability}.
{\em Existence:}
Let $(\beta^{**},\gamma^{**})$ be the minimum point of $\mathcal{L}_{A}(\beta,\gamma)$ in the region $\mathfrak{D} \cap  \ker (\tilde{\Gamma}^{\top})$, where the existence of this point is guaranteed by \cref{pro. minimum point of LA}.
By denoting 
$U^{**}\triangleq \frac{\partial \mathcal L (\beta^{**},\gamma^{**})}{\partial \beta}=\frac{\partial \mathcal L_{\mathcal A} (\beta^{**},\gamma^{**})}{\partial \beta} + U^o$  and $N^{**}\triangleq \left(\frac{\partial \mathcal L (\beta^{**},\gamma^{**})}{\partial \gamma}\right)^{\top}=\left(\frac{\partial \mathcal L_{\mathcal A} (\beta^{**},\gamma^{**})}{\partial \beta}\right)^{\top} + N^o$,
we can conclude that $(U^{**},N^{**})$ is a positive state (by the definition of the Legendre transformation) and satisfies
(by the minimum of $(\beta^{**},\gamma^{**})$)
\begin{equation}{\label{eq. proof of asymptotic stability 1}}
\left(
\begin{array}{c}
U^{**}-U^o \\ N^{**}-N^{o}
\end{array}
\right)
=
\nabla \mathcal{L}_{\mathcal{A}}(\beta^{**},\gamma^{**})
\in {\rm Im} \tilde{ \Gamma},
\end{equation}
which suggests $(U^{**},N^{**}) \in \mathcal{PS}(U^o,N^o)$. 
Note that if $\Delta \mathcal{U}_{\rm CR}\neq \mathbbold{0}^{r_{\rm CR}}$, then the condition $U^o=(N^o)^{\top}u(T_e)$, \cref{condition. 3}, and \eqref{eq. proof of asymptotic stability 1} implies $U^{**}=(N^{**})^{\top}u(T_e)$.
According to the Legendre transformation $\mathcal{L}(\beta,\gamma)$, we can further conclude that
\begin{equation*}
\nabla S_{\mathcal A}(U^{**}, N^{**}) =
\left(
\begin{array}{c}
\beta^{**} \\ \gamma^{**}
\end{array}
\right)
\in \ker (\tilde{\Gamma}^{\top}),
\end{equation*}
and, therefore, $(U^{**},N^{**})$ is a detailed balanced equilibrium in $\mathcal{PS}(U^o,N^o)$ (by \cref{cor. detailed balancing of each equilibrium}).

{\em Uniqueness:} Note that \cref{cor. detailed balancing of each equilibrium} suggests any positive detailed balanced equilibrium to be a stationary point of $S_{\mathcal{A}} (U,N)$ in its positive stoichiometric compatibility class.
Therefore, the uniqueness of the detailed balanced equilibrium follows immediately from the strict convexity of $S_{\mathcal{A}} (U,N)$.

{\em Asymptotic stability:} As mentioned above, any positive detailed balanced equilibrium is a stationary point of $S_{\mathcal{A}} (U,N)$ in its positive stoichiometric compatibility class.
Therefore, the strict convexity of $S_{\mathcal{A}} (U,N)$ suggests the function $S_{\mathcal{A}} (U,N)$ to be lower bounded in $\mathcal{PS}(U^o,N^o)$ with the minimum point at the unique detailed balanced equilibrium. 
Note that by \cref{thm. stability} and \cref{cor. detailed balancing of each equilibrium}, the function $S_{\mathcal{A}} (U,N)$ is strictly dissipative at any positive state other than detailed balanced equilibrium.
Thus, by the second Lyapunov method, the unique detailed balanced equilibrium is locally asymptotically stable. $\Box$

\end{document}